\documentclass[10pt,a4paper]{article}
\usepackage[latin1]{inputenc}

\usepackage{amsmath}
\usepackage{caption}
\usepackage{amsfonts}
\usepackage{amssymb}
\usepackage{graphicx}
\usepackage{amsthm}
\usepackage{enumerate}
\usepackage{setspace}
\usepackage{booktabs}
\usepackage{tabularx}
\usepackage{lineno} 
\usepackage{color} 
\usepackage{enumerate}
\usepackage{xcolor}

\usepackage{tikz}
\usetikzlibrary{arrows,shapes,snakes,automata,backgrounds,petri}

\newcounter{i}
\newcommand\numberthis{\addtocounter{equation}{1}\tag{\theequation}}

\usepackage[top=1.42in, bottom=1.42in, left=1.25in, right=1.25in]{geometry}

\newtheorem{theorem}{Theorem}[section]

\newtheorem{lemma}[theorem]{Lemma}

\newtheorem{proposition}[theorem]{Proposition}
\newtheorem{conjecture}[theorem]{Conjecture}
\newtheorem{question}[theorem]{Question}

\newtheorem{definition}[theorem]{Definition}
\newtheorem{procedure}[theorem]{Procedure}
\newtheorem{claim}[theorem]{Claim}
\newtheorem{subclaim}[theorem]{Subclaim}

\newenvironment{proofclaim}[1][]
    {\noindent \emph{Proof.} {}{#1}{}}{\hfill $\Diamond$\vspace{1em}}
    
\theoremstyle{plain} 
\newcommand{\thistheoremname}{}
\newtheorem{genericthm}[section]{\thistheoremname}

\newcommand{\sn}{N^{s}} 
\newcommand{\sd}{d^{s}}

\newcommand{\overbar}[1]{\mkern 1.5mu\overline{\mkern-1.5mu#1\mkern-1.5mu}\mkern 1.5mu}

\title{Colouring Graphs with Sparse Neighbourhoods:\\Bounds and Applications}
\author{Marthe Bonamy\thanks{CNRS, LaBRI, Universit\'e de Bordeaux, email: \texttt{marthe.bonamy@u-bordeaux.fr}}, Thomas Perrett\thanks{Technical University of Denmark, email: \texttt{tper@dtu.dk}. Supported by ERC Advanced Grant GRACOL, project number 320812.}, Luke Postle\thanks{University of Waterloo, email: \texttt{lpostle@uwaterloo.ca}. Partially supported by NSERC under Discovery Grant No. 2014-06162.}}
\date{\today}

\begin{document}
\maketitle
\begin{abstract}
Let $G$ be a graph with chromatic number $\chi$, maximum degree $\Delta$ and clique number $\omega$. Reed's conjecture states that $\chi \leq \lceil (1-\varepsilon)(\Delta + 1) + \varepsilon\omega \rceil$ for all $\varepsilon \leq 1/2$. It was shown by King and Reed that, provided $\Delta$ is large enough, the conjecture holds for $\varepsilon \leq 1/130,000$. In this article, we show that the same statement holds for $\varepsilon \leq 1/26$, thus making a significant step towards Reed's conjecture. We derive this result from a general technique to bound the chromatic number of a graph where no vertex has many edges in its neighbourhood. Our improvements to this method also lead to improved bounds on the strong chromatic index of general graphs. We prove that $\chi'_s(G)\leq 1.835 \Delta(G)^2$ provided $\Delta(G)$ is large enough.
\end{abstract}

\section{Introduction}

It is well known that the chromatic number $\chi(G)$ of a graph $G$ is bounded above by $\Delta(G) +1$, where $\Delta(G)$ denotes the maximum degree of $G$. Similarly, a trivial lower bound on $\chi(G)$ is given by the clique number $\omega(G)$, which is the largest number of pairwise adjacent vertices in $G$. In 1998, Reed conjectured that, up to rounding, the chromatic number of a graph is at most the average of these two bounds.

\begin{conjecture}\label{conj:Reeds}\emph{\cite{Reeds}}
If $G$ is a graph, then $\chi(G) \leq \lceil \frac12 (\Delta(G) + 1 + \omega(G)) \rceil$.
\end{conjecture}

As evidence for his conjecture, Reed proved that the chromatic number can be bounded above by a non-trivial convex combination of $\omega$ and $\Delta+1$.

\begin{theorem}\label{thm:KingReedEpsilon}\emph{\cite{Reeds}}
There exists $\varepsilon > 0$ such that for every graph $G$, we have $\chi(G) \leq \lceil (1-\varepsilon)(\Delta(G) + 1) + \varepsilon \omega(G)  \rceil$.
\end{theorem}

King and Reed~\cite{KingReedShort} subsequently gave a shorter proof of Theorem~\ref{thm:KingReedEpsilon} by exploiting a recent result of King~\cite{KingIndep} on independent sets hitting every maximal clique. Using King's result, it suffices to prove Theorem~\ref{thm:KingReedEpsilon} for graphs $G$ with clique number $\omega(G) \leq \frac{2}{3}(\Delta(G) +1)$. Given this fact, there are two main steps in the proof of King and Reed. The first is to show that if such a graph is also critical, then no neighbourhood contains many edges. More precisely, there exists $\delta >0$ such that every neighbourhood induces at most $(1-\delta){\Delta(G) \choose 2}$ edges. We say that such a graph is \emph{$\delta$-sparse}. The second step is to invoke the naive colouring procedure and the probabilistic method to colour the graph. Indeed, using these techniques, it can be shown that a $\delta$-sparse graph is $(1-\varepsilon)(\Delta(G) + 1)$-colourable for some $\varepsilon > 0$ depending on $\delta$. This completes the proof.

Seeking only a short proof of Theorem~\ref{thm:KingReedEpsilon}, King and Reed did not optimise the two steps of their method. Approximately, they find that $\delta = 1/160$ and $\varepsilon < 1/320e^6$ suffice. However, since Reed's Conjecture is equivalent to proving Theorem~\ref{thm:KingReedEpsilon} for $\varepsilon \leq 1/2$, it is natural to ask if one can increase the value of $\varepsilon$ obtained. It would suffice to provide an improved answer to any of the two following questions. Recall first that a graph $G$ is \emph{ $(k+1)$-critical} if $G$ is not $k$-colorable but every proper subgraph of $G$ is.

\begin{question}\label{ques:1}
Let $G$ be a $\lfloor (1-\varepsilon)(\Delta(G)+1) \rfloor + 1$-critical graph with $\omega(G) \leq (1-\alpha)(\Delta(G)+1)$. What is the largest $\delta = \delta(\varepsilon,\alpha)$, such that $G$ is $\delta$-sparse?
\end{question}

\begin{question}\label{ques:2}
Let $G$ be a $\delta$-sparse graph. What is the largest $\varepsilon = \varepsilon(\delta)$ such that $\chi(G) \leq (1-\varepsilon)(\Delta+1)$?
\end{question}

\subsection{Main Results}

In this paper we improve on the best known results for both of these questions. In fact, we prove results in the context of list colouring, a generalization of colouring. A \emph{list assignment} is a function that to each vertex $v\in V(G)$ assigns a nonempty set $L(v)$ of colours. An \emph{$L$-colouring} is a coloring $\phi$ of $G$ such that $\phi(v)\in L(v)$ for every $v\in V(G)$. A \emph{$k$-list-assignment} is a list assignment $L$ such that $|L(v)|\ge k$ for every $v\in V(G)$. A graph $G$ is \emph{$k$-list-colourable} if $G$ has an $L$-coloring for every $k$-list-assignment $L$. The \emph{list chromatic number} of $G$, denoted $\chi_{\ell}(G)$ is the minimum $k$ such that $G$ is $k$-list-colourable. We say a graph $G$ is \emph{$L$-critical} with respect to a list assignment $L$ if $G$ does not have an $L$-colouring but every proper subgraph of $G$ does.

In response to Question~\ref{ques:1}, we prove the following theorem.

\begin{theorem}\label{thm:density}
Let $\varepsilon, \alpha > 0$ such that $\varepsilon < \frac{\alpha}{2}$. If $G$ is $L$-critical with respect to some $ \lceil (1-\varepsilon)(\Delta(G) +1) \rceil$-list-assignment $L$ and $\omega(G) \leq (1-\alpha)(\Delta(G)+1)$, then $G$ is $\frac{(\alpha-2 \varepsilon)^2}2$-sparse.
\end{theorem}

Note that this implies the same result for $\lfloor (1-\varepsilon)(\Delta(G) +1) \rfloor + 1$-critical graphs. King and Reed~\cite{KingReedShort} showed that if $G$ is a $\lfloor (1-\varepsilon)(\Delta(G) +1) \rfloor + 1$-critical graph with clique number $\omega(G) \leq \frac{2}{3}(\Delta(G) + 1)$, then $G$ is $\delta$-sparse provided $\delta < \frac{1}{4}(\frac{1}{6}-\varepsilon)^2$. Setting $\alpha = 1/3$ in Theorem~\ref{thm:density}, our bound gives $\delta = 2(\frac{1}{6}-\varepsilon)^2$, an eightfold improvement.

Question~\ref{ques:2} is a well studied problem. Molloy and Reed~\cite{MolloyReedSCI} proved that, for $\delta \in [0, 0.9]$, one may take $\varepsilon(\delta) =  0.0238\delta$ provided that the maximum degree is large enough. More recently, with the same conditions, Bruhn and Joos~\cite{BruhnJoos} improved this to $\varepsilon(\delta) = 0.1827 \delta - 0.0778\delta^{3/2}$. These bounds are approximations of more complicated expressions, see~\cite{MolloyReedSCI} and~\cite{BruhnJoos} respectively. Both of these results are proved using a single application of the naive colouring procedure, a randomised colouring technique which generates a partial proper colouring of a $\delta$-sparse graph. In this article, we develop an iterative version and using this we improve the bound of Bruhn and Joos by a factor of $\sqrt{e} \approx 1.6487$ as follows.

\newcounter{DeltaSparsity}
\setcounter{DeltaSparsity}{\thei}
\addtocounter{i}{1}
\begin{theorem}\label{thm:SparsityApprox}
Let $G$ be a $\delta$-sparse graph with $\delta \in [0,0.9]$, and let $\varepsilon = 0.3012 \delta - 0.1283 \delta^{3/2}$. There exists $\Delta_\theDeltaSparsity(\delta)$ such that if $\Delta(G) > \Delta_\theDeltaSparsity(\delta)$, then $\chi(G) \leq \chi_{\ell}(G) \leq (1-\varepsilon)(\Delta(G)+1)$.
\end{theorem}

In fact, we prove Theorem~\ref{thm:SparsityApprox} in the setting of correspondence colouring defined in Section~\ref{sec:sparsity}, a generalization of list colouring. The use of correspondence colouring allows us to simplify some of the intricacies in the proof and is quite natural in this setting.

This paper is not the first to consider an iterative application of the naive colouring procedure. Indeed, the notable result of Johansson~\cite{Johansson}, which states that triangle-free graphs satisfy $\chi(G) \leq O(\Delta(G)/\log \Delta(G))$ is proved in this way, see also~\cite{MolloyReedBook}. Triangle-free graphs behave particularly nicely with respect to an iterative version because, for any partial colouring, the subgraph induced by the uncoloured vertices is still triangle-free. We should briefly remark however that the method of Johannson~\cite{Johansson} is somewhat different in the sense that the procedure is only applied to a fraction of the vertices in each step. In this case the technique is often called the \emph{semi random method} or \emph{R\"odl nibble} and can be traced back to~\cite{AKS,Rodl}.

In this paper, we show that for $\delta$-sparse graphs, the naive colouring procedure can generate a partial colouring with the additional property that the uncoloured subgraph $G'$ is almost $\delta$-sparse (see Lemma~\ref{lem:stablesparsity}). This is the key which allows us to apply the procedure iteratively to the uncoloured subgraph. In addition, the probability that a vertex remains coloured is about $e^{-1/2}$ (see Proposition~\ref{prop:prob}) and hence the probability a vertex is in $G'$ is about $p=1-e^{-1/2}$. After one iteration, Bruhn and Joos had shown that the difference between the maximum degree of $G'$ and the resulting list sizes had decreased by at least $(0.1827 \delta - 0.0778\delta^{3/2})\Delta(G)$; if that was the initial difference, then we could greedily colour $G'$ to finish. However, given the key lemma that $G'$ is almost $\delta$-sparse, we may apply the procedure again. In each step, we accrue a new savings proportional to the current maximum degree. Terminating this procedure ad infinitum would result in roughly the following savings:

\begin{align*}
(0.1827 \delta - 0.0778\delta^{3/2})\Delta(G)(1+p+p^2 + p^3 + \ldots) &= (0.1827 \delta - 0.0778\delta^{3/2})\Delta(G)\frac{1}{1-p}\\
 &=e^{1/2}(0.1827 \delta - 0.0778\delta^{3/2})\Delta(G)\\
&\approx 0.3012 \delta - 0.1283 \delta^{3/2}
\end{align*}

Of course, we cannot carry out this iteration indefinitely, but after four iterations, we have saved as much as claimed in Theorem~\ref{thm:SparsityApprox}. For technical reasons, we adopt a different perspective in the proof of Theorem~\ref{thm:SparsityApprox}, wherein we study the ratio of maximum degree to list size and show that as long as this ratio is at most that of Theorem~\ref{thm:SparsityApprox}, then the ratio will slowly decrease after each iteration until it falls below $1$ whereupon we finish by colouring greedily.

By using Theorem~\ref{thm:density} and Theorem~\ref{thm:SparsityApprox} together with the technique of King and Reed, we obtain that the $\varepsilon$-version of Reed's Conjecture holds for $\varepsilon = 1/26$.

\newcounter{Deltamain}
\setcounter{Deltamain}{\thei}
\addtocounter{i}{1}
\begin{theorem}\label{thm:MainThmReeds}
There exists $\Delta_\theDeltamain >0$ such that if $G$ is a graph of maximum degree $\Delta > \Delta_\theDeltamain$ and clique number $\omega$, then $\chi(G) \leq \lceil \frac{25}{26}(\Delta + 1) + \frac1{26} \omega \rceil$.
\end{theorem}

\subsection{The Strong Chromatic Index}

The \emph{strong chromatic index}, $\chi'_s(G)$, of a graph $G$ is defined as the least integer $k$ for which there exists a $k$-colouring of $E(G)$ such that edges at distance at most $2$ receive different colours. Equivalently, $\chi'_s(G) = \chi (L^2(G))$, where $L^2(G)$ denotes the square of the line graph of $G$. Since $\Delta(L^2(G)) < 2\Delta(G)^2$, the trivial upper bound on the chromatic number gives that $\chi'_s(G) \leq 2\Delta(G)^2$. However Erd\H{o}s and Ne\v{s}et\v{r}il conjectured a much stronger upper bound, see~\cite{conj1.25}.

\begin{conjecture}\label{conj:1.25}
If $G$ is a graph, then $\chi'_s(G) \leq 1.25 \Delta(G)^2$.
\end{conjecture}

If true, this bound would be tight. Indeed, if $G_k$ denotes the graph obtained from a $5$-cycle by blowing up each vertex into $k$ vertices, then $\Delta(G_k) = 2k$ and $L^2(G_k)$ is a clique with $5 k^2 = 1.25 \Delta(G_k)^2$ vertices. Figure~\ref{fig:blowup} depicts the graph $G_3$.

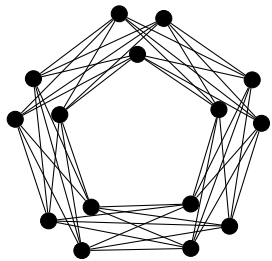
\begin{figure}[h]
\centering
\begin{tikzpicture}[scale=1.7]
\tikzstyle{whitenode}=[draw=white,circle,fill=white,minimum size=0pt,inner sep=0pt]
\tikzstyle{blacknode}=[draw,circle,fill=black,minimum size=6pt,inner sep=0pt]
\draw (0,0) node[whitenode] (a1) {};
\foreach \i in {2,3,4,5} {\pgfmathparse{\i-1} \let\j\pgfmathresult \draw (a\j) ++(72*\j-72:1) node[whitenode] (a\i) {};}
\foreach \i in {1,2,3,4,5} {\foreach \j in {1,2,3} {\draw (a\i) ++(120*\j+72*\i-2*72:0.2) node[blacknode] (b\i\j) {};}}
\foreach \i/\k in {1/2,2/3,3/4,4/5,5/1} {\foreach \j in {1,2,3} {\foreach \l in {1,2,3} {\draw (b\i\j) edge node {} (b\k\l);}}}
\end{tikzpicture}\caption{A blow-up of the $5$-cycle.}\label{fig:blowup}
\end{figure}

In 1997, Molloy and Reed made the first step towards Conjecture~\ref{conj:1.25}. They showed that for all graphs $G$, the graph $L^2(G)$ is a subgraph of a graph $H$ such that $\Delta(H) = 2\Delta(G)^2$ and $H$ is $1/36$-sparse. Thus the naive colouring procedure guarantees that $H$ (and hence $G$) can be coloured with $(1-\varepsilon)(2\Delta(G)^2 + 1)$ colours for some $\varepsilon >0$.

\begin{theorem}\emph{\cite{MolloyReedSCI}}\label{thm:SCIepsilon}
There exists $\varepsilon >0$ such that if $G$ is a graph with sufficiently large maximum degree $\Delta$, then $\chi'_s(G) \leq (1-\varepsilon) \cdot 2 \Delta^2$.
\end{theorem}

With $\delta = 1/36$ and their colouring procedure, the value of $\varepsilon$ that Molloy and Reed obtain is approximately $0.0238 \cdot \frac{1}{36} \approx 0.0007$. Bruhn and Joos~\cite{BruhnJoos} improved the bound on the neighbourhood sparsity and showed that $L^2(G)$ is asymptotically $1/4$-sparse. With $\delta = 0.24$, say, and their colouring procedure, they deduce Theorem~\ref{thm:SCIepsilon} for $\varepsilon = 0.1827 \cdot 0.24 - 0.0778 \cdot 0.24^{3/2} \approx 0.0347$. This gives the following.

\begin{theorem}\emph{\cite{BruhnJoos}}\label{thm:BruhnJoosSCI}
If $G$ is a graph of sufficiently large maximum degree $\Delta$, then $\chi'_s(G) \leq 1.93 \Delta^2$. 
\end{theorem}

In this article we improve the bound in Theorem~\ref{thm:BruhnJoosSCI}. To do this we first show that one only needs to colour a subgraph $F$ of $L^2(G)$ consisting of high degree vertices with many neighbours of high degree. This idea resembles the notion that one need only colour a critical subgraph of $L^2(G)$. We then show that $F$ admits a much better bound on its neighbourhood sparsity than $L^2(G)$. Combined with Theorem~\ref{thm:SparsityApprox}, we obtain the following result.

\begin{theorem}\label{thm:MainThmStrongEdge}
If $G$ is a graph of sufficiently large maximum degree $\Delta$, then $\chi'_s(G) \leq 1.835 \Delta^2$. 
\end{theorem}

\subsection{Outline of the Paper}

In Section~\ref{sec:density} we deal with Question~\ref{ques:1} and prove Theorem~\ref{thm:density}. In Section~\ref{sec:sparsity} we address Question~\ref{ques:2}. We recall the naive colouring procedure and develop an iterative version. We then derive Theorem~\ref{thm:SparsityApprox} as a consequence. Section~\ref{sect:strongedge} is devoted to the strong chromatic index and the proof of Theorem~\ref{thm:MainThmStrongEdge}. Finally, in Section~\ref{sec:ReedsConj}, we prove Theorem~\ref{thm:MainThmReeds}.

For standard definitions and graph theoretic notation, we refer the reader to Diestel~\cite{diestel}.

\section{A Density Lemma}\label{sec:density}

In this section we prove Theorem~\ref{thm:density}, which guarantees that a graph that is critical with respect to some $k$-list-assignment is $\delta$-sparse, for some $\delta$ depending on $k$ and the clique number of $G$. To do this, we first show that if $G$ is an $L$-critical graph with respect to some $k$-list-assignment $L$, then the minimum degree of an induced subgraph of $G$ cannot be too large.

\begin{proposition}\label{prop:minDegree}
If $G$ is an $L$-critical graph with respect to some $k$-list-assignment $L$, then for all induced subgraphs $H$ of $G$, we have $\delta(H) < \Delta(G) - k + \chi_\ell(H)$.
\end{proposition}

\begin{proof}
Suppose for a contradiction that $H$ is an induced subgraph of $G$ with $\delta(H) \geq \Delta(G)-k+\chi_\ell(H)$. Let $G' = G - V(H)$ and note that for every vertex $v \in V(H)$, we have $d_{G'}(v) \leq \Delta(G) - \delta(H) \leq k - \chi_\ell (H)$. Since $G$ is $L$-critical, $G'$ has an $L$-colouring $\phi$. Now to each vertex $v \in V(H)$, assign a list of colours $L'(v)$, defined by $L'(v) = L(v) \setminus \{\phi(u) \: | \: u \in N_{G'}(v)\}$. For each $v \in V(H)$, we have $|L'(v)| \geq k - d_{G'}(v) \geq \chi_\ell (H) $. Hence $\phi$ can be extended to an $L$-colouring of $G$, a contradiction.
\end{proof}

The bound in Proposition~\ref{prop:minDegree} exhibits an awkward dependence on $\chi_\ell(H)$, and so we first derive an upper bound on this parameter. Note that we let $\overbar{G}$ denote the complement of $G$. One can easily guarantee a large matching in the complement of a graph if the clique number is small.

\begin{proposition}\label{prop:antimatching}
For every graph $G$, $\overbar{G}$ has a matching of size at least $\lceil \frac12 (|V(G)| - \omega(G)) \rceil$.
\end{proposition}

\begin{proof}
If $M$ is a maximal matching in $\overbar{G}$, then $G - V(M)$ is a clique. Thus $|V(G)| - 2|M| \leq \omega(G)$.
\end{proof}

We make use of the following classical result of Erd\H{o}s, Rubin and Taylor~\cite{ERT}.  

\begin{theorem}\label{thm:ERT}\emph{\cite{ERT}}
Let $r$ be an integer. If $G$ is a complete $r$-partite graph where each partition class contains at most two vertices, then $\chi_\ell (G) = r$.
\end{theorem}

Using Proposition~\ref{prop:antimatching} and Theorem~\ref{thm:ERT}, we can derive the desired bound.

\begin{proposition}\label{prop:listChrom}
If $G$ is a graph then $\chi_\ell (G) \leq \lfloor \frac12 (|V(G)| + \omega(G)) \rfloor$.
\end{proposition}

\begin{proof}
By Proposition~\ref{prop:antimatching}, the graph $G$ has an antimatching $M$ of size $\lceil \frac12 (|V(G)| - \omega(G)) \rceil$. Let $G'$ denote the graph with vertex set $V(G)$ and edge set ${V(G) \choose 2} \setminus M$. Note that $G'$ satisfies the conditions in Theorem~\ref{thm:ERT} with $r = \omega(G) + \lfloor \frac12 (|V(G)| - \omega(G)) \rfloor = \lfloor \frac12 (|V(G)| + \omega(G)) \rfloor$. Thus, by Theorem~\ref{thm:ERT}, we have $\chi_\ell(G') = r$. Now, since $G \subseteq G'$, we have $\chi_\ell(G) \leq \chi_\ell(G') = r$ as desired.
\end{proof}

Let $G$ be a graph and $A$ be a subset of $V(G)$ with $A = \{v_1, \dots, v_r\}$. We say that $v_1, \dots, v_r$ is a \emph{minimum-degree ordering} of $A$ if $v_i$ is a vertex of minimum degree in the subgraph $G[\{v_i, \dots, v_r\}]$, for all $i \in \{1, \dots, r\}$. We use this ordering to derive a first bound on $\delta$.

\begin{lemma}\label{lem:gap^2-gap*saved}
Let $G$ be a graph of maximum degree $\Delta$ and clique number $\omega$. If $G$ is $L$-critical with respect to some $k$-list-assignment $L$, then for every vertex $v \in V(G)$, we have $${\Delta \choose 2} - |E(G[N(v)])| \geq \frac12 \cdot {2k - \Delta - \omega + 1 \choose 2}.$$
\end{lemma}

\begin{proof}
Let $v$ be a vertex of $G$ with $d(v) = r$, and let $D(v) = {\Delta \choose 2} - |E(G[N(v)])|$. Also, let $H$ denote the graph formed from $G[N(v)]$ by adding $\Delta - r$ independent vertices. We do this so as to compare $|E(G[N(v)])|$ more easily with ${\Delta \choose 2}$, as it is the maximum number of edges in the neighborhood of a vertex of degree $\Delta$. Finally, let $v_1, \dots, v_\Delta$ be a minimum-degree ordering of $V(H)$, and set $H_i = H[\{v_i, \dots, v_\Delta\}]$ for $i \in \{1, \dots, \Delta\}$. Clearly, we have $D(v) = \sum_{i=1}^{\Delta} \left( |V(H_i)| - 1 -d_{H_i}(v_i) \right)$.
For $i \in \{1, \dots, \Delta - r\}$, the vertex $v_i$ is isolated, and thus $d_{H_i}(v_i) = 0$. On the other hand, for $i \in \{\Delta - r +1, \dots, \Delta\}$, the vertex $v_i$ has degree $d_{H_i}(v_i) = \delta(H_i) < \Delta - k + \chi_\ell (H_i)$ by Proposition~\ref{prop:minDegree},  so we have
$$D(v) \geq \sum_{i=1}^{\Delta} \max\{0,  |V(H_i)| - (\Delta - k) -  \chi_\ell (H_i)\}.$$

By Proposition~\ref{prop:listChrom}, we have $\chi_\ell(H_i) \leq  \frac12 (|V(H_i)| + \omega(G))$. Furthermore, $|V(H_i)| = \Delta - i + 1$ for each $i \in \{1, \ldots, \Delta \}$. Thus, we have:
\begin{align*}
D(v) &\geq \sum_{i=1}^{\Delta} \max \left\{0,  \frac{\Delta - i+1}{2} - \frac{\omega}{2} - (\Delta - k)\right\}\\
 &= \frac12 \sum_{i=1}^{\Delta} \max \left\{0,  2k-\Delta-\omega - i+1 \right\}. \numberthis \label{eq:truncate_sum}
\end{align*}

The second term in the maximum of~\eqref{eq:truncate_sum} eventually becomes negative when $i > 2k - \Delta - \omega + 1$. Because of the maximum, we may truncate the sum and deduce that

\begin{align*}
D(v) &\geq \frac12 \cdot \sum_{i=1}^{2k-\Delta-\omega} \max \left\{0,  2k-\Delta-\omega - i+1 \right\}\\
&= \frac12 \cdot \sum_{j=1}^{2k- \Delta - \omega} j \\
& = \frac12 \cdot {2k - \Delta - \omega + 1 \choose 2}.
\end{align*}
\end{proof}

We can now prove Theorem~\ref{thm:density}.

\begin{proof}[Proof of Theorem~\ref{thm:density}]
Let $k= \lceil (1-\varepsilon)(\Delta(G) +1) \rceil$. By Lemma~\ref{lem:gap^2-gap*saved}, we have for every vertex $v\in V(G)$ that
\begin{align*}
{\Delta \choose 2} - |E(G[N(v)])| &\geq \frac12 \cdot {2k - \Delta - \omega + 1 \choose 2}\\ 
& = \frac12 \cdot {2\lceil (1-\varepsilon)(\Delta +1) \rceil - \Delta - (1-\alpha)(\Delta+1) + 1 \choose 2}\\
& \ge \frac12 \cdot {(\alpha-2\varepsilon)(\Delta+1) + 1 \choose 2}\\
& \ge \frac12 (\alpha-2\varepsilon)^2 {\Delta \choose 2}.
\end{align*}
Hence $G$ is $\frac{(\alpha-2\varepsilon)^2}2$-sparse.
\end{proof}

\section{A Sparsity Lemma}\label{sec:sparsity}

\subsection{The Naive Colouring Procedure}\label{subsec:ncp}

The \emph{naive colouring procedure} is a well studied technique which generates a partial proper $k$-colouring of a graph $G$. In the context of graph colourings it was first used by Kahn~\cite{KahnList}, though it had already appeared in a more abstract setting~\cite{AKS}. We refer the reader to~\cite{MolloyReedBook} for a survey on further applications of the technique. In its simplest form, the naive colouring procedure consists of the following two steps.

\begin{enumerate}
\item To each vertex $u \in V(G)$, assign a colour chosen uniformly at random from $\{1,\dots, k\}$.
\item If $u$ and $v$ are adjacent vertices with the same colour, then uncolour both $u$ and $v$.
\end{enumerate}

Let $G$ be a graph and $k$ be an integer with $k < \Delta(G)+1$. If no vertex of $G$ has too many edges in its neighbourhood, then one can show that with positive probability, the partial $k$-colouring generated by the above procedure has the property that vertices of large degree see many repeated colours in their neighbourhoods. To be more precise, let $\textrm{Col}(u)$ denote the number of coloured vertices in $N(u)$ and let $\textrm{Dist}(u)$ denote the number of \emph{distinct} colours amongst the colours of the vertices in $N(u)$. If there are repeated colours in $N(u)$, then clearly $\textrm{Col}(u) > \textrm{Dist}(u)$. The following proposition states that if the difference is large enough, then such a partial colouring can be extended to a colouring of the whole graph in an efficient way.

\begin{proposition}\label{prop:partial}
Let $G$ be a graph and $k$ be an integer such that $k < \Delta(G) + 1$. If there is a partial proper $k$-colouring of $G$ such that for every vertex $u \in V(G)$, we have $\textrm{Col}(u) - \textrm{Dist}(u) \geq d(u)+1-k$, then $G$ has a $k$-colouring.
\end{proposition}

\begin{proof}
Let $u \in V(G)$ be an uncoloured vertex. The number of uncoloured neighbours of $u$ is precisely $d(u) - \textrm{Col}(u)$. The number of colours in $\{1, \dots, k\}$ which do not appear in $N(u)$ is $k - \textrm{Dist}(u) \geq d(u) - \textrm{Col}(u)+1$. It remains to list colour the uncoloured subgraph $G'$, where every vertex $u \in V(G')$ has a list of size at least one greater than $d_{G'}(u)$. Such a colouring can be constructed greedily.
\end{proof}

It is hard to analyse the expectation of the random variable $\textrm{Col}(u) - \textrm{Dist}(u)$. However, by inclusion-exclusion, it is easy to see that $\textrm{Col}(u) - \textrm{Dist}(u) \geq P_u - T_u$, where $P_u$ and $T_u$ denote the number of pairs and triples of vertices in $N(u)$ which are all coloured the same and all remain coloured after the procedure. When computing the expectation of $P_u$ and $T_u$, it is convenient to assume that the graph in question is $\Delta$-regular. Indeed, this is no restriction, since if $G$ is a graph of maximum degree $\Delta$, then $G$ may be embedded in a $\Delta$-regular graph $G'$ by iterating the following process. Take two copies of $G$ and add edges between corresponding vertices of degree less than $\Delta$. Note that $\chi(G') = \chi(G)$ and if $G$ is $\delta$-sparse, then so is $G'$. In this way, we will frequently assume that the graph under consideration is $\Delta$-regular.

Once the expectations of $P_u$ and $T_u$ have been calculated, one can show that they are concentrated about their expectations. In other words, the probability that $P_u - T_u$ is far from its expectation is very small. The Lov\'{a}sz Local Lemma can then be applied to ensure that this is the case for every $u \in V(G)$.

\newtheorem*{lllem}{Lov\'asz Local Lemma}
\newcommand{\Dep}{\textnormal{Dep}}

\begin{lllem}
	Let $p \in [0, 1)$, $d$ a postive integer, and $\mathcal{B}$ a finite set of (bad) events such that for every $B \in \mathcal{B}$,
	\begin{itemize}
		\item $\Pr[B] \le p$, and
		\item There exists a set of events $\Dep(B) \subseteq \mathcal{B}$ of size at most $d$ such that $B$ is mutually independent of $\mathcal{B} \setminus \Dep(B)$. 
	\end{itemize}
	If $4pd \le 1$, then there exists an outcome in which none of the events in $B$ occur.
\end{lllem}

In this paper we show that the naive colouring procedure can be iterated. More precisely, we prove that if $G$ is a $\delta$-sparse graph, then after a single application of the procedure the graph induced by the uncoloured vertices retains some of the sparsity of the original graph. Thus we can apply the procedure again to the uncoloured subgraph. In order to show that the sparsity is retained, we first show that with positive probability, the set of uncoloured vertices behaves somewhat randomly. The precise condition that we require is the following.
 
\begin{definition}\label{def:mu-quasirandom}
Let $\mu \in [0,1]$, $G$ be a graph with maximum degree $\Delta$, and $A \subseteq V(G)$. We say that $G[A]$ is a \emph{$\mu$-quasirandom} subgraph of $G$ if for every pair of not necessarily distinct vertices $u, v \in V(G)$, we have $$||N(u)\cap N(v) \cap A|-\mu |N(u)\cap N(v)|| \leq \sqrt{\Delta}(\log \Delta)^5.$$
\end{definition}

Note that for $u=v$, the condition in Definition~\ref{def:mu-quasirandom} reduces to $|d_A(u) - \mu d(u)| \leq \sqrt{\Delta}(\log \Delta)^5$. To show that the uncoloured subgraph is a $\mu$-quasirandom subgraph of $G$, we track more random variables which count the number of uncoloured vertices in the common neighbourhood of two vertices. These random variables will also be shown to be highly concentrated, and so we can add the corresponding bad events to our previous application of the Lov\'{a}sz Local Lemma.

\subsection{Correspondence Colouring}

Any iterative application of the naive colouring procedure necessitates the introduction of lists of colours. This is because in each step, some colours are forbidden at a vertex $v$, namely those which have been assigned to the neighbours of $v$ in a previous application. In analysing the procedure, a technical issue arises due to the fact that the probability a vertex keeps a particular colour in its list may vary depending on the vertex and the colour. Previously, this issue has been dealt with by introducing extra vertices, or coin flips, to equalise the probabilities. 

Here, we use a generalisation of list colouring called \emph{correspondence colouring}, introduced by Dvo\v{r}\'ak and the third author in~\cite{correspondence} (and sometimes referred to as $\textnormal{DP}$-coloring). As well as proving a more general statement, the use of correspondence coloring automatically equalises the probabilities, and thus simplifies the proof. Here is the definition we use which is equivalent to but slightly different from the definitions given elsewhere.

\begin{definition}\emph{\cite{correspondence}}\label{def:correspondence}
Let $G$ be a graph, and let $\vec{G}$ be an arbitrary orientation of $G$.
\begin{itemize}
\item A \emph{correspondence assignment} $C$ of $G$ is a function defined on $V(G) \cup E(\vec{G})$ as follows: To each vertex $u \in V(G)$, $C$ assigns a set $C(u) \subseteq \mathbb{N}$, and to each edge $uv\in E(\vec{G})$, $C$ assigns an injective partial function $C_{uv}:C(u)\to C(v)$ such that $C_{vu}=C_{uv}^{-1}$ for every edge $uv \in \vec{G}$. 
\item If each $C(u)$ has size at least $k$, then $C$ is a \emph{$k$-correspondence assignment} for $G$.
\item A \emph{$C$-colouring} of $G$ is a function
 $f:V(G)\to\mathbb{N}$ such that $f(u) \in C(u)$ for every $u \in V(G)$, and for every edge $uv\in E(\vec{G})$, either $f(u)\not\in\text{dom}(C_{uv})$ or $C_{uv}(f(u))\neq f(v)$.
\item The \emph{correspondence chromatic number} of $G$, denoted $\chi_c(G)$, is the smallest integer $k$ such that $G$ is $C$-colourable
for every $k$-correspondence assignment $C$.
\end{itemize}
\end{definition}

We say that the function $C_{uv}$ assigned to the edge $uv$ is \emph{total} if $\text{dom}(C_{uv}) = C(u)$. Note that there is no requirement that functions in the definition above are total. Hence the following definition.

\begin{definition}
Let $G$ be a graph and $C$ be a correspondence assignment of $G$. We say $C$ is \emph{total} if $C_{uv}$ and $C_{vu}$ are total for every edge $uv$ of $G$.
\end{definition} 

Note that if $C$ is total and $G$ is connected, then $|C(u)|=|C(v)|$ for every pair of vertices $u,v\in V(G)$. We remark if $C$ is a correspondence assignment of a graph $G$ such that $|C(u)|=|C(v)|$ for every pair of vertices $u,v\in V(G)$, then we will often extend $C$ to a total correspondence assignment $C'$ by arbitrarily extending each function $C_{uv}, uv \in E(G)$ to be total. Clearly, if $G$ is $C'$-colourable, then $G$ is also $C$-colourable. 

\begin{definition}
Let $G$ be a graph and let $C$ be a total correspondence assignment of $G$. If $uv \in E(G)$, $c_1\in C(u)$, $c_2\in C(v)$, then we say $c_1$ and $c_2$ \emph{correspond} under $C$ if $C_{uv}(c_1) = c_2$, or equivalently, $C_{vu}(c_2) = c_1$. If the correspondence assignment is clear from the context, then we simply say that $c_1$ and $c_2$ \emph{correspond}. 
\end{definition}

Note that Proposition~\ref{prop:partial} is still valid for correspondence colouring.

We now state precisely the variant of the naive colouring procedure that we use. Let $C$ be a $k$-correspondence assignment.

\begin{procedure}\label{proc:ncp}
Suppose $G$ is a graph and $C$ is a correspondence assignment for $G$. We generate a partial $C$-colouring $f$ as follows.
\begin{enumerate}[Step 1:]
\item Assign each vertex $u \in V(G)$ a colour $f_1(u)$ chosen uniformly at random from $C(u)$.\label{step:colour}
\item For every edge $uv \in E(G)$, pick an end $D(uv)$ uniformly at random, that is $D(uv)=u$ with probability $\frac12$ and $D(uv)=v$ with probability $\frac12$.\label{step:direct}
\item\label{step:Uncolour} For each vertex $u\in V(G)$, let $f(u)=f_1(u)$ if and only if for every edge $uv\in E(G)$, at least one of the following hold: $C_{uv}(f_1(u))\ne f_1(v)$ or $D(uv)=v$. (Equivalently, uncolour $u$ if there exists an edge $uv\in E(G)$ such that $C_{uv}(f_1(u))=f_1(v)$ and $D(uv)=u$.)
\end{enumerate}
\end{procedure}

We remark that the uncolouring method used here in Steps~\ref{step:direct} and~\ref{step:Uncolour} was also used by Bruhn and Joos~\cite{BruhnJoos}. Before analysing the procedure, we note the following fundamental fact.

\begin{proposition}\label{prop:prob}
Let $G$ be a $\Delta$-regular graph and let $C$ be a total $k$-correspondence assignment of $G$. For every vertex $u \in V(G)$, the probability that $u$ is coloured after an application of Procedure~\ref{proc:ncp} (that is $f(u)=f_1(u)$) is $(1-\frac{1}{2k})^\Delta$.
\end{proposition}

\begin{proof}
Let $K$ be the event that $f(u)=f_1(u)$. For each neighbour $v$ of $u$, let $U_v$ be the event that $C_{uv}(f_1(u)) = f_1(v)$ and $D(uv)=u$. Now by definition, $\mathbb{P}[K] = \mathbb{P}[\bigcap_{v\in N(u)} \overline{U_v}]$. Since these events are independent, we find that $\mathbb{P}[K] = \prod_{v \in N(u)}\mathbb{P}[\overline{U_v}].$

Note that $\mathbb{P}[U_v] = \mathbb{P}[C_{uv}(f_1(u)) = f_1(v)] \times \mathbb{P}[D(uv)=u],$
since the events are independent. Since all correspondences are total, $\mathbb{P}[C_{uv}(f_1(u)) = f_1(v)] = \frac{1}{k}$. Furthermore, $\mathbb{P}[D(uv)=u] = \frac12$. Hence $\mathbb{P}[U_v] = \frac{1}{2k}$ and $\mathbb{P}[\overline{U_v}] = 1 - \frac{1}{2k}$.

Thus $\mathbb{P}[K] = \prod_{v \in N(u)}(1-\frac{1}{2k}) = (1-\frac{1}{2k})^{|N(u)|}$.  As $G$ is $\Delta$-regular, $|N(u)|=\Delta$. Hence $\mathbb{P}[K] = (1-\frac{1}{2k})^\Delta$ as desired.
\end{proof}

We are ready to prove the key lemma of this section. The result is similar to Lemma 7 in Bruhn and Joos~\cite{BruhnJoos}, however we extend it to correspondence colouring, and we ensure that the uncoloured vertices induce a $\mu$-quasirandom subgraph.

\newcounter{DeltaColouring}
\setcounter{DeltaColouring}{\thei}
\addtocounter{i}{1}
\begin{lemma}\label{lem:colouring}
Let $G$ be a $\Delta$-regular $\delta$-sparse graph and let $C$ be a $k$-correspondence assignment for $G$. Also let $\gamma > 0$ satisfy $$\gamma < \frac{\Delta\delta}{2k}e^{-\frac{\Delta}{k}} - \frac{\Delta^2\delta^{\frac{3}{2}}}{6k^2} e^{-\frac{7\Delta}{8k}}.$$ There exists an integer $\Delta_\theDeltaColouring(\delta,\gamma)$ such that if $\Delta \geq \Delta_\theDeltaColouring(\delta,\gamma)$, then there is a $\mu$-quasirandom subgraph $G'$ of $G$, and a $k'$-correspondence assignment $C'$ of $G'$ such that any $C'$-colouring of $G'$ extends to a $C$-colouring of $G$, where $\mu = 1 - (1-\frac{1}{2k})^\Delta$ and $ k' \geq k- (1-\mu - \gamma)\Delta$.
\end{lemma}

\begin{proof}[Proof of Lemma~\ref{lem:colouring}]

We may assume that for each vertex $u \in V(G)$, the set $C(u)$ has size precisely $k$ (by restricting to an arbitrary subset of $C(u)$ of size $k$). Furthermore, we may assume that $C$ is a total correspondence assignment (by extending, for each edge $uv$, the function $C_{uv}$ to an arbitrary total function and setting $C_{vu}= C^{-1}_{uv}$). Note the latter two assumptions only restrict the possible set of $C$-colourings.

Now consider an application of Procedure~\ref{proc:ncp} to the graph $G$, which produces a partial $C$-colouring $f$ of $G$. Let $G'$ be the subgraph of $G$ induced by the uncoloured vertices, and let $C'$ be the correspondence assignment obtained from $C$ as follows: For each $u \in V(G')$, let $C'(u) := C(u) \setminus \{C_{vu}(f(v)): v\in N_G(u)\setminus V(G')\}$. To every edge $uv$ in $E(G')$, let $C'$ assign the map $C'_{uv}$, where $C'_{uv}$ is the restriction of $C_{uv}$ to $C'(u)$ and $C'(v)$. 

We set $k' = \min_{u \in V(G')} |C'(u)|$. Note that every $C'$-colouring $\phi'$ of $G'$ can be extended to a $C$-colouring $\phi$ of $G$ by letting $\phi(v) = \phi'(v)$ if $v\in V(G')$ and $\phi(v)=f(v)$ otherwise. Moreover, we could truncate each $C'(u)$ to an arbitrary subset of size $k'$ restricting further the possible $C'$-colorings. However, this is not technically needed since the definition of $k$-correspondence assignment we use requires only lists of size at least (not necessarily equal to) $k$. 

It remains to show that both of the following hold with high probability: $V(G')$ is $\mu$-quasirandom; and $k'\geq k - (1-\mu - \gamma)\Delta$. 

To this end we define a collection of events and random variables. Firstly, for each pair of vertices $u,v \in V(G)$ such that the distance from $u$ to $v$ is at most $2$, we define a random variable $N_{u,v}$ by $N_{u,v} = |N(u) \cap N(v) \cap V(G')|$. In particular, for a vertex $u \in V(G)$, we have $N_{u,u} = d_{V(G')}(u)$. Let $B_{u,v}$ be the event that $$|N_{u,v} - \mu |N(u)\cap N(v)|| \geq \sqrt{\Delta}(\log \Delta)^5.$$ 

We show that the probabilities of all these bad events are small in the following two claims.

\begin{claim}\label{cl:Nuv}
For every $u,v \in V(G)$, we have $\mathbb{P}[B_{u,v}] \leq \Delta^{-\frac12 \log \log \Delta}.$
\end{claim}
\begin{proofclaim}
By Proposition~\ref{prop:prob}, we have $\mathbb{E}[N_{u,v}] = \mu|N(u) \cap N(v)|$. Thus $\mathbb{P}[B_{u,v}]=\mathbb{P}[|N_{u,v} - \mathbb{E}[N_{u,v}]| \geq \sqrt{\Delta}(\log \Delta)^5]$. In Section~\ref{sec:concentration} we argue that the random variable $N_{u,v}$ is highly concentrated about its expectation. More precisely, it follows from Lemma~\ref{lem:NuvConcentration} that $$\mathbb{P}[|N_{u,v} - \mathbb{E}[N_{u,v}]| \geq \sqrt{\Delta}(\log \Delta)^5] \leq \Delta^{-\frac{1}{2}\log \log \Delta}.$$ Hence the conclusion.
\end{proofclaim}

For every vertex $u \in V(G)$, let 
$$P_u := |\{v_1v_2\in E(\overline{G}[N(u)]): C_{v_1u}(f_1(v_1)) = C_{v_2u}(f_1(v_2)), v_1,v_2\notin V(G')\}|.$$ 
\noindent That is, $P_u$ denotes the number of non-adjacent pairs of vertices in $N(u)$ whose colours under $f$ correspond to the same colour at $u$. For a graph $H$, let $T(H)$ denote the set of triangles of $H$. For every vertex $u\in V(G)$, let 

$$T_u := |\{v_1v_2v_3\in T(\overline{G}[N(u)]): C_{v_1u}(f_1(v_1)) = C_{v_2u}(f_1(v_2)) = C_{v_3u}(f_1(v_3)), v_1,v_2,v_3\notin V(G')\}|.$$ 
\noindent That is, $T_u$ denotes the number of non-adjacent triples of vertices in $N(u)$ whose colours under $f$ correspond to the same colour at $u$. 

For convenience, for each $u\in V(G)$, let $\delta_u > \delta$ be a fixed constant such that $N(u)$ induces precisely $(1 - \delta_u) {\Delta \choose 2}$ edges.

We begin by finding a lower bound on the expectation of $P_u$ as follows.

\begin{claim}\label{cl:P}
For each vertex $u\in V(G)$, we have
$$\mathbb{E}[P_u]\geq (1-o(1)) \cdot \delta_u \frac{\Delta^2}{2k}e^{-\frac{\Delta}{k}},$$
where $o(1)$ denotes a function that tends to $0$ as $\Delta$ tends to infinity.
\end{claim}
\begin{proofclaim}

Let $c\in C(u)$. First let $v_1$ and $v_2$ be non-adjacent neighbours of $u$. Let $c_1 = C_{uv_1}(c)$ and $c_2=C_{uv_2}(c)$.

Note that
$$\mathbb{P}[f(v_1)=c_1,f(v_2)=c_2] = \mathbb{P}[f_1(v_1)=c_1,f_1(v_2) = c_2] \cdot \mathbb{P}[v_1,v_2\not\in V(G')| f_1(v_1)=c_1,f_1(v_2)=c_2].$$
Yet 
$$\mathbb{P}[f_1(v_1)=c_1, f_1(v_2) = c_2] = \mathbb{P}[f_1(v_1)=c_1] \cdot \mathbb{P}[f_1(v_2)=c_2] = \frac{1}{k^2},$$ 
since the events are independent. 

Thus we proceed to calculate $\mathbb{P}[v_1,v_2\not\in V(G')| f_1(v_1)=c_1, f_1(v_2)=c_2]$ as follows. For each $xy\in E(G)$, let $U_{x,y}$ be the event that $C_{xy}(f_1(x))=f_1(y)$ and $D(xy)=x$ (that is the event that $y$ `uncolours' $x$). Note for each $xy\in E(G)$, $\mathbb{P}[U_{x,y}] = \frac{1}{2k}$ and hence $\mathbb{P}[\overline{U_{x,y}}]= 1 - \frac{1}{2k}$. 

\begin{align*}
\mathbb{P}[v_1,& v_2\not\in V(G')| f_1(v_1)=c_1, f_1(v_2)=c_2]\\
&= \mathbb{P}\Big[ \Big( \bigcap_{w\in N(v_1)} \overline{U_{v_1,w}} \Big) \cap \Big (\bigcap_{x\in N(v_2)} \overline{U_{v_2,x}} \Big)\Big]   \\ 
&= \prod_{w\in N(v_1)\setminus N(v_2)} \mathbb{P}[ \overline{U_{v_1,w}}] \ \times \prod_{x\in N(v_2)\setminus N(v_1)} \mathbb{P}[\overline{U_{v_2,x}}] \ \times \prod_{y\in N(v_1)\cap N(v_2)} \mathbb{P}[\overline{U_{v_1,y}} \cap \overline{U_{v_2,y}}].
\end{align*}

 For each $y\in N(v_1)\cap N(v_2)$, we have by the union bound that $\mathbb{P}[U_{v_1,y} \cup U_{v_2,y}] \le \frac{1}{k}$ and hence $\mathbb{P}[\overline{U_{v_1,y}}\cap \overline{U_{v_2,y}}] \ge 1 - \frac{1}{k}$.

Let $|N(v_1)\cap N(v_2)| = \ell$. Since $G$ is $\Delta$-regular, we have that $|N(v_1)\cap N(v_2)|=|N(v_2)\cap N(v_1)| = \Delta-\ell$. Hence 
$$\mathbb{P}[v_1, v_2 \notin V(G')| f_1(v_1)=f_1(v_2)=c] \geq \left( 1-\frac{1}{2k} \right)^{2\Delta-2\ell} \cdot \left( 1-\frac{1}{k} \right)^\ell \geq \left(1-\frac{1}{k}\right)^\Delta.$$

Thus for each $c\in C(u)$
$$\mathbb{P}[C_{v_1u}(f(v_1))=C_{v_2u}(f(v_2))=c] \geq \frac{1}{k^2} \cdot \left(1-\frac{1}{k}\right)^\Delta.$$

Since $|C(u)|=k$, we find that 

$$\mathbb{P}[C_{v_1u}(f(v_1))=C_{v_2u}(f(v_2))] \geq \frac{1}{k} \cdot \left(1-\frac{1}{k}\right)^\Delta.$$

Yet

$$\mathbb{E}[P_u] = \sum_{v_1v_2 \in E(\overline{N(u)})} \mathbb{P}[C_{v_1u}(f(v_1))=C_{v_2u}(f(v_2))].$$

As there are precisely $\delta_u {\Delta \choose 2}$ non-adjacent pairs in $N(u)$, we conclude that
\begin{align*}
\mathbb{E}[P_u]&\geq \delta_u {\Delta \choose 2} \cdot \frac{1}{k}\cdot \left(1-\frac{1}{k}\right)^\Delta\\
&\geq (1-o(1)) \cdot \delta_u \frac{\Delta^2}{2k}e^{-\frac{\Delta}{k}},
\end{align*}
as desired, where the last inequality follows because the two inequalities that $\gamma$ is assumed to satisfy imply that $k = \Theta(\Delta)$.
\end{proofclaim}

We now compute an upper bound on the expectation of $T_u$, the number of non-adjacent triples of vertices in $N(u)$ whose colours under $f$ correspond to the same colour at $u$, as follows.

\begin{claim}\label{cl:T}
For each vertex $u\in V(G)$, we have
$$\mathbb{E}[T_u]\leq \frac{\Delta^3\delta^{\frac{3}{2}}}{6k^2} e^{-\frac{7\Delta}{8k}}.$$
\end{claim}
\begin{proofclaim}
Let $c\in C(u)$. Let $v_1,v_2,v_3 \in T(\overline{G}[N(u)])$. For $i\in \{1,2,3\}$, let $\ell_i$ denote the number of vertices in $N(v_1) \cup N(v_2) \cup N(v_3)$ that have precisely $i$ neighbours in $\{v_1,v_2,v_3\}$. Note that $\ell_1+2\ell_2+3\ell_3=3\Delta$. 

We now proceed with an analysis similar to that for the pairs. Let $c_1=C_{uv_1}(c)$, $c_2=C_{uv_2}(c)$ and $c_3=C_{uv_3}(c)$. 

Since 

$$\mathbb{P}[f_1(v_1)=c_1,f_1(v_2)=c_2,f_1(v_3) = c_3]=\frac1{k^3},$$

it suffices to compute 
$$ \mathbb{P}[v_1,v_2,v_3\not\in V(G')| f_1(v_1)=c_1,f_1(v_2)=c_2,f_1(v_3)=c_3].$$

For each $y$ with precisely two neighbours in $\{v_1,v_2,v_3\}$, say $y\in N(v_1)\cap N(v_2)$ and $y \not\in N(v_3)$, we have $\mathbb{P}[U_{v_1,y} \cup U_{v_2,y}] \ge \frac{3}{4k}$. Indeed, $\mathbb{P}[U_{v_1,y} \cap U_{v_2,y}]=\mathbb{P}[f_1(y)=C_{v_1y}(c_1)= C_{v_2y}(c_2)] \cdot \mathbb{P}[D(yv_1)=v_1 \textrm{ and } D(yv_2)=v_2]$. Note that $\mathbb{P}[D(yv_1)=v_1 \textrm{ and } D(yv_2)=v_2]=\frac34$. The value of $\mathbb{P}[f_1(y)=C_{v_1y}(c_1)= C_{v_2y}(c_2)]$ is either $0$ (if $C_{v_1y}(c_1) \neq C_{v_2y}(c_2)$) or $\frac1{k}$ (if $C_{v_1y}(c_1) = C_{v_2y}(c_2)$). In both cases, $\mathbb{P}[U_{v_1,y} \cap U_{v_2,y}]\leq \frac1{4k}$, hence $\mathbb{P}[U_{v_1,y} \cup U_{v_2,y}]\geq \mathbb{P}[U_{v_1,y}]+\mathbb{P}[U_{v_2,y}] - \frac1{4k}=\frac3{4k}$. It follows that $\mathbb{P}[\overline{U_{v_1,y}}\cap \overline{U_{v_2,y}}] \le 1 - \frac{3}{4k}$.

Similarly, for each $y\in N(v_1)\cap N(v_2) \cap N(v_3)$, we have $\mathbb{P}[U_{v_1,y} \cup U_{v_2,y}\cup U_{v_3,y}] \ge \frac{7}{8k}$ and hence $\mathbb{P}[\overline{U_{v_1,y}}\cap \overline{U_{v_2,y}}\cap \overline{U_{v_3,y}}] \le 1 - \frac{7}{8k}$.

We find that the probability that $v_1,v_2,v_3 \in V(G)\setminus V(G')$ is at most $\left(1-\frac{1}{2k}\right)^{\ell_1}\cdot \left(1-\frac{3}{4k}\right)^{\ell_2}\cdot \left(1-\frac{7}{8k}\right)^{\ell_3}$. Since we can check that $(1-\frac{1}{2k})^3\leq (1-\frac{7}{8k})$ and $(1-\frac{3}{4k})^2\leq (1-\frac{7}{8k})$, we find that the probability that $v_1,v_2,v_3 \in V(G)\setminus V(G')$ is at most:

$$\left(1-\frac{1}{2k}\right)^{\ell_1}\cdot \left(1-\frac{3}{4k}\right)^{\ell_2}\cdot \left(1-\frac{7}{8k}\right)^{\ell_3}\leq  \left(1-\frac{7}{8k}\right)^{\ell_1/3+2\ell_2/3+\ell_3}=\left(1-\frac{7}{8k}\right)^{\Delta}.$$

A result of Rivin~\cite{Rivin} states that every graph with $\delta_u {\Delta \choose 2}$ edges contains at most $\frac{\delta^{3/2}_u\Delta^3}{6}$ triangles. Thus, for large enough $\Delta$,
$$\mathbb{E}[T_u]\leq \frac{\delta^{3/2}_u\Delta^3}{6}
\cdot k \cdot \frac{1}{k^3} \cdot \left(1-\frac{7}{8k}\right)^{\Delta} \leq  \frac{\Delta^3\delta^{\frac{3}{2}}}{6k^2} e^{-\frac{7\Delta}{8k}},$$
as desired.
\end{proofclaim}

Now, using linearity of expectation and Claims~\ref{cl:P} and~\ref{cl:T}, we have
\begin{align*}
\mathbb{E}[P_u - T_u] &\geq (1-o(1)) \cdot \frac{\Delta^2\delta_u}{2k}e^{-\frac{\Delta}{k}} -  \frac{\Delta^3\delta_u^{\frac{3}{2}}}{6k^2} e^{-\frac{7\Delta}{8k}}\\
&\geq (1-o(1)) \left( \frac{\Delta\delta}{2k}e^{-\frac{\Delta}{k}} - \frac{\Delta^2\delta^{\frac{3}{2}}}{6k^2} e^{-\frac{7\Delta}{8k}}\right)\Delta.
\end{align*}

As discussed after Proposition~\ref{prop:partial}, the value of $P_u-T_u$ is a lower bound on the number of repeated colours. Let $A_u$ be the event that $$P_u - T_u \leq \left( 1 - \frac{1}{\log \Delta}\right) \left( \frac{\Delta\delta}{2k}e^{-\frac{\Delta}{k}} - \frac{\Delta^2\delta^{\frac{3}{2}}}{6k^2} e^{-\frac{7\Delta}{8k}}\right)\Delta.$$ 

\begin{claim}\label{cl:Au}
For every $u \in V(G)$, we have $\mathbb{P}[A_u] \leq 2\Delta^{-\frac12 \log \log \Delta}.$
\end{claim}
\begin{proofclaim}
We argue in Section~\ref{sec:concentration} that the random variables $P_u$ and $T_u$ are highly concentrated about their expectations. More precisely, Lemmas~\ref{lem:PuConcentration} and~\ref{lem:TuConcentration} state that for large enough $\Delta$ we have that 

$$\mathbb{P}[|P_u - \mathbb{E}[P_u]| \geq \sqrt{\Delta} \log^4 \Delta ] \leq \Delta^{-\frac12 \log \log \Delta}$$ and $$\mathbb{P}[|T_u - \mathbb{E}[T_u]| \geq \sqrt{\Delta} \log^5 \Delta ] \leq \Delta^{-\frac12 \log \log \Delta}.$$

Thus, it follows from Claims~\ref{cl:P} and~\ref{cl:T} that for $\Delta$ large enough, we have $\mathbb{P}[A_u] \leq 2 \Delta^{-\frac12 \log \log \Delta}$ as desired.
\end{proofclaim}

Let $\mathcal{A}$ and $\mathcal{B}$ denote the set of events of the form $A_u$ and $B_{u,v}$ respectively. For $x,y\in V(G)$, let $d_G(x,y)$ denote the distance from $x$ to $y$ in $G$. Note that $A_u$ only depends on random variables in $\{f_1(w) : d_G(u,w)\le 2\} \cup \{D(wx): d_G(u,w)\le 1\}$ and similarly $B_{u,v} $ depends only on random variables in $\{f_1(w) : d_G(u,w)\le 2\}\cup \{D(wx): d_G(u,w)\le 1$. 

Let $d= \Delta^9$. A routine calculation show that for each $x\in V(G)$ and integer $i\ge 0$, we have $|\{y: d(x,y)\le i\}| \le \Delta^{i+1}.$ 

For each $u\in V(G)$, let $\Dep(A_u) := \{ A_v: d_G(v,u)\le 4\}\cup \{B_{v,w}: d_G(v,u)\le 4\}$. Note that, by the mutual independence principle, $A_u$ is mutually independent of all events in $(\mathcal{A}\cup \mathcal{B}) \setminus \Dep(A_u)$. Furthermore $|\Dep(A_u)| \le \Delta^5 + \Delta^5 \cdot \Delta^3 \le \Delta^9$ since $\Delta\ge 2$. Hence $|\Dep(A_u)| \le d$. By Claim~\ref{cl:Au}, it follows that $\mathbb{P}[A_u] \le 2\Delta^{-\frac12 \log \log \Delta} \le \frac{1}{4d}$ where the last equality follows since $\Delta$ is large enough. 

Similarly for each $u,v$ such that $d_G(u,v)\le 2$, let $\Dep(B_{u,v}) := A_w: d_G(w,u)\le 4\}\cup \{B_{w,x}: d_G(w,u)\le 4\}.$ Note that, by the mutual independence principle, $B_{u,v}$ is mutually independent of all events in $(\mathcal{A}\cup \mathcal{B}) \setminus \Dep(B_{u,v})$. Furthermore as above $|\Dep(B_{u,v})| \le  \Delta^5 + \Delta^5\cdot\Delta^3 \le \Delta ^9 = d$. By Claim~\ref{cl:Au}, it follows that $\mathbb{P}[A_u] \le \Delta^{-\frac12 \log \log \Delta} \le \frac{1}{4d}$ where the last equality follows since $\Delta$ is large enough. 

The Lov\'{a}sz Local Lemma then implies that with positive probability none of the events in $\mathcal{A} \cup \mathcal{B}$ occur. Since no event in $\mathcal{B}$ occurs, the uncoloured subgraph $G'$ is $\mu$-quasirandom. In particular, every vertex $u \in V(G)$ has at most $(1-\mu)\Delta + \sqrt{\Delta}(\log \Delta)^5$ coloured neighbours. Similarly, since no event in $\mathcal{A}$ occurs, we have that, if $\Delta$ is large enough, then $P_u - T_u \geq \gamma \Delta + \sqrt{\Delta}(\log \Delta )^5$ for every $u \in V(G')$.

Now, the $k'$-correspondence assignment $C'$ of $G'$ satisfies $$|C'(u)| \geq k - ((1-\mu)\Delta + \sqrt{\Delta}(\log \Delta)^5) + P_u - T_u$$ at every vertex $u \in V(G')$. Thus, we have $k' \geq k - (1-\mu - \gamma)\Delta$ as desired.
\end{proof}

\subsection{Concentration Details}\label{sec:concentration}

In this section we prove the concentration results required in the proof of Lemma~\ref{lem:colouring}. Our main tool will be a modified version of Talagrand's inequality, developed by Bruhn and Joos~\cite{BruhnJoos}.

Consider a random variable $X$ determined by a set of independent trials. If changing the outcome of a small number of trials does not affect $X$ very much, then a well known concentration inequality may apply. Unfortunately, in the naive colouring procedure, changing the colour of one vertex can have a large effect. Indeed changing the colour of a vertex $u$ may cause all vertices in $N(u)$ to lose their colour during Step~\ref{step:Uncolour} of Procedure~\ref{proc:ncp}. However such an outcome is very unlikely, since it requires that the colours assigned to the vertices in $N(u)$ all correspond to the same colour at $u$. 

Bruhn and Joos~\cite{BruhnJoos} developed a version of Talagrand's Inequality capable of handling such outcomes. To describe it, let $\Omega$ be a product space of discrete probability spaces, and define a set $\Omega^* \subseteq \Omega$ of \emph{exceptional} outcomes. We say that \emph{$X$ has downward $(s,c)$-certificates} if for every $t>0$, and for every $\omega\in\Omega\setminus\Omega^*$ there is an index set $I$ of size at most $s$ so that $X(\omega')< X(\omega)+t$ for every $\omega'\in\Omega\setminus\Omega^*$ where the restrictions $\omega|_I$ and ${\omega'}|_I$ differ in less than $t/c$ coordinates. 

In other words, for each non-exceptional outcome, there is a small index set which can guarantee that the random variable $X$ is not too much larger for similar outcomes. We can now state the theorem of Bruhn and Joos.

\begin{theorem}\label{thm:BruhnJoosTalagrand}\emph{\cite{BruhnJoos}}
Let $((\Omega_i,\Sigma_i,\mathbb{P}_i))_{i=1}^n$ be discrete probability spaces, $(\Omega,\Sigma,\mathbb{P})$ be their product space, and let $\Omega^*\subset \Omega$ be a set of exceptional outcomes.
Let $X:\Omega\to\mathbb R$ be a random variable, let $M=\max\{\sup |X|,1\}$, and let $c\geq 1$.
If $\mathbb{P}[\Omega^*]\leq M^{-2}$
and $X$ has downward $(s,c)$-certificates, then for $t> 50 c\sqrt{s}$, $$\mathbb{P}[|X- \mathbb{E}[X]|\geq t] \leq 4e^{-\frac{t^2}{16c^2s}}+4\mathbb{P}[\Omega^*].$$
\end{theorem}

For each vertex $v\in V(G)$, let $\Omega_v$ denote the discrete probability space that is selecting a colour $f_1(v)$ from $C(v)$ uniformly at random. For each edge $uv\in E(G)$, let $\Omega_{uv}$ denote the discrete probability space that is selecting an end $D(uv)$ uniformly at random from $\{u,v\}$. Let $\Omega$ denote the product probability space $\prod_{v\in V(G)} \Omega_v \times \prod_{uv\in E(G)} \Omega_{uv}$. Thus each outcome $\omega \in \Omega$ is indexed by $V(G) \cup E(G)$. 

For each vertex $v\in V(G)$, let $Q_v$ be the set of outcomes $\omega \in \Omega$ such that there exists a subset $S$ of $N(v)$, $|S|\ge \log \Delta$, and $c\in C(v)$ such that $C_{wv}(f_1(w)) = c$ for all $w\in S$ (that is at least $\log \Delta$ vertices in $N(v)$ have colours corresponding to the same colour at $v$).

Let $u \in V(G)$ be a fixed vertex. We define the exceptional outcomes 
$$\Omega^* := \bigcup_{v: d(u,v)\le 2} Q_v.$$

\begin{lemma}\label{lem:OmegaStar}
For large enough $\Delta,$ 
$$\mathbb{P}[\Omega^*]\leq  \Delta^{-\frac23 \log \log \Delta}.$$
\end{lemma}
\begin{proof}
First we calculate $\mathbb{P}[Q_v]$. This calculation can be found in~\cite[p. 18 (arXiv version)]{BruhnJoos} and trivially generalises to correspondence colouring. Hence we have the following:
$$\mathbb{P}[Q_v] \leq \Delta^2 (\frac{e \Delta}{k \log \Delta})^{\log \Delta}.$$
As the number of vertices at distance at most two from a vertex is at most $\Delta^2+1$, we have
\begin{equation}
\mathbb{P}[\Omega^*]\leq (\Delta^2+1)\Delta^2 \left(\frac{e\Delta}{k \log \Delta}\right)^{\log \Delta} = \Delta^{6+ \log\left(\frac{\Delta}{k}\right) - \log\log \Delta} \leq \Delta^{-\frac23 \log \log \Delta},
\end{equation}

\noindent as desired, where the last inequality follows since $\Delta$ is large enough. Note that the middle term is $\Delta^2$ times bigger than the probability obtained in~\cite{BruhnJoos}, but that increase is negligible given how fast it decreases in $\Delta$.
\end{proof}

\begin{proposition}\label{prop:PuCertificate}
For each $u \in V(G)$, the random variable $P_u$ has downward $(s, c)$-certificates where $s = 3\Delta$ and $c = \log^2 \Delta$.
\end{proposition}

The proof is almost identical to that of Bruhn and Joos~\cite[Lemma 7]{BruhnJoos} except that, since we deal with correspondence colouring, it is possible that a single vertex $v$ affects the colours of many vertices in $N(u)$, all of which correspond to different colours at $u$.

\begin{proof}
Let $\omega \in \Omega \setminus \Omega^*$. For every vertex $v\in N(u)\cap V(G')$, let $a_v$ denote a neighbour $w$ of $v$ such that $C_{wv}(f_1(w)(\omega))=f_1(v)(\omega)$ and $D(wv)(\omega) = v$ (such exist since $v\in V(G')$). Let $I= \big( \bigcup_{v\in N(u)} \Omega_v \big) \cup \big( \bigcup_{v\in N(u)\cap V(G')} (\Omega_{a_v} \cup \Omega_{va_v} ) \big)$. Note that $|I| \leq 3\Delta$.

To prove that $P_u$ has downward $(3\Delta, \log^2 \Delta)$-certificates it now suffices to prove the following claim.

\begin{claim}
For every $t > 0$ and $\omega' \in \Omega \setminus \Omega^*$ such that that $P_u(\omega') \geq P_u(\omega) + t$, then $\omega|_I$ and $\omega'|_I$ differ in at least $\frac{t}{\log^3 \Delta}$ coordinates. 
\end{claim}
\begin{proofclaim}
First we characterize how the coordinates in $I$ may differ between $\omega$ and $\omega'$ as follows. Let 
$$A_1 := \{ v \in N(u): \Omega_v(\omega) ) \ne \Omega_v(\omega')\},$$ that is the set of neighbours of $u$ that have different colours under $f_1$ in $\omega$ versus $\omega'$. Similarly let 
$$A_2 := \{ v \in N(u): \Omega_{va_v}(\omega) \ne \Omega_{va_v}(\omega') \},$$
that is the neighbours of $u$ where $D(va_v)$ differs in $\omega$ and $\omega'$. Finally let 
$$A_3 := \{w\in V(G): \exists v\in N(u)\cap V(G'), w=a_v, \Omega_w(\omega) \ne \Omega_w(\omega')\},$$ that is the vertices $w\in V(G)$ for which $f_1(w)$ differs in $\omega$ and $\omega'$ and are also an $a_v$ for some $v\in N(u)\cap V(G')$. Note that $\omega|_I$ and $\omega'|_I$ differ in at most $|A_1|+|A_2|+|A_3|$ coordinates. 

Now let $$P_u^w := |\{wx\in E(\overline{G}[N(u)]): C_{wu}(f_1(w)) = C_{xu}(f_1(x)), w,x\notin V(G')\}|,$$ that is the number of pairs counted in $P_u$ in which $w$ appears. Let 
$$B=\{ v\in N(u): v\in V(G'(\omega))\setminus V(G'(\omega'))\},$$
that is the set of neighbours of $u$ that are in $V(G')$ in $\omega$ but not in $V(G')$ in $\omega'$.
Now
$$P_u(\omega') \le P_u(\omega) + \sum_{w\in A_1\cup B} P_u^w(\omega').$$

For each vertex $v\in V(G)$ and colour $c \in C(v)$, define 
$$N_{v,c}:= \{x\in N(v): C_{xv}(f_1(x))=c\},$$
\noindent that is the set of vertices $x\in N(v)$ whose colour in $f_1$ corresponds to colour $c$ at $x$.
Since $\omega, \omega' \not\in \Omega^*$, we have that $\omega,\omega' \not\in Q_v.$ This implies that for each vertex $v\in V(G)$ and colour $c\in C(v)$, we have 
$$|N_{v,c}| \le \log \Delta.$$ 
Note that if $w\in N(u)$ and we let $c=C_{wu}(\Omega_w(\omega'))$, then $P_u^w(\omega') \le |N_{u,c}| \le \log \Delta$, hence
$$P_u(\omega') \le P_u(\omega) + (|A_1|+|B\setminus A_1|)\log \Delta.$$
Yet
$$B\setminus A_1 \subseteq A_2 \cup \big( \bigcup_{w\in A_3} N_{w,\Omega_(w)}  (\omega) \ \big),$$
hence
$$|B\setminus A_1| \le |A_2| + |A_3|\log \Delta.$$
Combining, we have
$$t \le P_u(\omega')-P_u(\omega) \le (|A_1|+|B\setminus A_1|)\log \Delta \le (|A_1|+|A_2|+|A_3|)\log^2 \Delta.$$
Hence the number of coordinates in which $\omega|_I$ and $\omega'|_I$ differ is at least $\frac{t}{\log^2\Delta} = \frac{t}{c}$ as desired.
\end{proofclaim}

\end{proof}

\begin{lemma}\label{lem:PuConcentration}
If $\Delta$ is large enough, then $\mathbb{P}[|P_u - \mathbb{E}[P_u]| \geq \sqrt{\Delta} \log^4 \Delta ] \leq \Delta^{-\frac12 \log \log \Delta}.$
\end{lemma}
\begin{proof}
We will apply Theorem~\ref{thm:BruhnJoosTalagrand} with $t=\sqrt{\Delta} \log^4 \Delta$, $s = 3\Delta$ and $c = \log^2 \Delta$ but first we check that the hypotheses of Theorem~\ref{thm:BruhnJoosTalagrand} are satisfied. 

Note that by Proposition~\ref{prop:PuCertificate}, $P_u$ has downward $(s,c)$-certificates. Next note  that $M = \sup P_u \leq \Delta^2$. By Lemma~\ref{lem:OmegaStar}, $\mathbb{P}[\Omega^*] \leq  \Delta^{-\frac23 \log \log \Delta}$ which is at most $\Delta^{-4}$ when $\Delta$ is large enough. Thus we have $\mathbb{P}[\Omega^*] \leq \Delta^{-4} \leq M^{-2}$. Hence all of the hypotheses of Theorem~\ref{thm:BruhnJoosTalagrand} are satisfied. 

Applying Theorem~\ref{thm:BruhnJoosTalagrand} with the parameters above, we conclude that for large enough $\Delta$, we have $$\mathbb{P}[|P_u - \mathbb{E}[P_u]| \geq \sqrt{\Delta} \log^4 \Delta ] \leq 4 \Delta^{- \frac{1}{48}{\log \Delta}} + \Delta^{-\frac23 \log \log \Delta},$$

which is at most $ \Delta^{-\frac12 \log \log \Delta} $ for large enough $\Delta$.
\end{proof}

In an analogous way one can show that the random variable $T_u$ is concentrated about its expectation. The only difference in the argument is that there could be up to ${\log \Delta \choose 3}$ triples of vertices in $N(u)$ which correspond to a fixed colour at $u$. Nevertheless, taking $t$ and $c$ to be $\log \Delta $ times larger than for $P_u$ above we obtain the following from Theorem~\ref{thm:BruhnJoosTalagrand}.

\begin{lemma}\label{lem:TuConcentration} If $\Delta$ is large enough, then $\mathbb{P}[|T_u - \mathbb{E}[T_u]| \geq \sqrt{\Delta} \log^5 \Delta ] \leq  \Delta^{-\frac12 \log \log \Delta}.$
\end{lemma}

For the random variable $N_{u,v}$, $u,v \in V(G)$, we can take $c=\log \Delta$, $s=3 \Delta$ and $t= \sqrt{\Delta} \log^2 \Delta$. An argument analogous to that of Proposition~\ref{prop:PuCertificate} shows that $N_{u,v}$ has downward $(s,c)$-certificates. Then Theorem~\ref{thm:BruhnJoosTalagrand} implies the following.

\begin{lemma}\label{lem:NuvConcentration} If $\Delta$ is large enough, then $\mathbb{P}[|N_{u,v} - \mathbb{E}[N_{u,v}]| \geq \sqrt{\Delta} \log^5 \Delta ] \leq  \Delta^{-\frac12 \log \log \Delta}.$
\end{lemma}

\subsection{Iterating the Procedure}

We now argue that given the properties of the colouring obtained after applying Lemma~\ref{lem:colouring}, the graph induced by the uncoloured vertices retains some of the sparsity of the original graph.

\newcounter{DeltaStableSparsity}
\setcounter{DeltaStableSparsity}{\thei}
\addtocounter{i}{1}
\begin{lemma}\label{lem:stablesparsity}
Let $\delta, \mu > 0$, let $G$ be a graph of maximum degree $\Delta$ such that every neighbourhood induces at most $(1-\delta){\Delta \choose 2}$ edges, and let $G'$ be a $\mu$-quasirandom subgraph of $G$. For every $\delta' < \delta$, there exists $\Delta_\theDeltaStableSparsity(\mu,\delta,\delta')$ such that if $\Delta \geq \Delta_\theDeltaStableSparsity(\mu,\delta,\delta')$, then every neighbourhood of $G'$ induces at most $(1-\delta'){\Delta(G') \choose 2}$ edges.
\end{lemma}

\begin{proof}
Let $u \in V(G')$, and for simplicity let $N'(u) = N(u) \cap V(G')$ and $d'(u) = d_{V(G')}(u)$. If $S$ is a set of vertices, we also write $E(S)$ to mean $E(G[S])$. Counting the edges induced by $N'(u)$, we have $2 |E(N'(u))| =  \sum_{v \in N'(u)} d_{N'(u)}(v)$. For any $v \in N'(u)$, since $G'$ is $\mu$-quasirandom and $d_{N'(u)}(v) = |N(u) \cap N(v) \cap V(G')|$, we have
\begin{equation}\label{eq:stabledegree}
d_{N'(u)}(v) \leq \mu d_{N(u)}(v) + \sqrt{\Delta}(\log \Delta )^5.
\end{equation}
Thus we have
\begin{align*}
2 |E(N'(u))| & \leq \sum_{v \in N'(u)} \left(\mu d_{N(u)}(v) + \sqrt{\Delta}(\log \Delta)^5\right)\\
& \leq \mu\sum_{v \in N'(u)} d_{N(u)}(v) + \Delta\sqrt{\Delta}(\log \Delta)^5.
\end{align*}
Rewriting the sum we have 
\begin{align*}
\sum_{v \in N'(u)} d_{N(u)}(v) &= \sum_{v \in N'(u)} \sum_{w \in N(u) \cap N(v)}1\\
& = \sum_{w \in N(u)} \sum_{v \in N'(u) \cap N(w) }1\\
& =  \sum_{w \in N(u)} d_{N'(u)}(w),
\end{align*}
so another application of~\eqref{eq:stabledegree} gives
\begin{align*}
2|E(N'(u))| &\leq \mu \sum_{w \in N(u)} \left(\mu d_{N(u)}(w)+ \sqrt{\Delta}(\log \Delta)^5\right)+ \Delta\sqrt{\Delta}(\log \Delta)^5\\
&= \mu^2 \sum_{w \in N(u)} d_{N(u)}(w) +  \mu\Delta \sqrt{\Delta}(\log \Delta)^5+ \Delta\sqrt{\Delta}(\log \Delta)^5\\
& \leq 2\mu^2 |E(N(u))| + 2\Delta\sqrt{\Delta}(\log \Delta)^5.
\end{align*}
Since every neighbourhood of $G$ induces at most $(1-\delta){\Delta \choose 2}$ edges, we have $$|E(N'(u))| \leq \mu^2 (1-\delta){\Delta \choose 2} + \Delta\sqrt{\Delta}(\log \Delta)^5,$$ and since $\mu^2 {\Delta \choose 2} \leq {\mu\Delta \choose 2} + \mu \Delta$ for any $\mu > 0$, we have $$|E(N'(u))| \leq (1-\delta){\mu\Delta \choose 2} + 2\Delta\sqrt{\Delta}(\log \Delta)^5.$$ Because $G'$ is a $\mu$-quasirandom subgraph of $G$, we have that $\mu \Delta \leq \Delta(G') + \sqrt{\Delta}(\log \Delta)^5$, so

\begin{align*}
|E(N'(u))| &\leq (1-\delta){ \Delta(G') + \sqrt{\Delta}(\log \Delta)^5 \choose 2} +  2\Delta\sqrt{\Delta}(\log \Delta)^5\\
& \leq (1-\delta) \left[ { \Delta(G') \choose 2} + { \sqrt{\Delta}(\log \Delta)^5\choose 2}+ \Delta(G')\sqrt{\Delta}(\log \Delta)^5 \right] + 2\Delta\sqrt{\Delta}(\log \Delta)^5
\end{align*}
Thus we have $|E(N'(u))|  \leq (1-\delta) { \Delta(G') \choose 2} + O(\Delta\sqrt{\Delta}(\log \Delta)^5)$. Finally, for any $\delta' < \delta$ we have $|E(N'(u))|  \leq (1-\delta') { \Delta(G') \choose 2}$ provided that $\Delta$ is large enough.
\end{proof}

We are now able to prove the main Theorem of this section.

\newcounter{DeltaMainThm}
\setcounter{DeltaMainThm}{\thei}
\addtocounter{i}{1}
\begin{theorem}\label{thm:MainThmCorr}
Let $\varepsilon, \delta >0$ be such that $\varepsilon<0.5$ and $$\varepsilon< e^{\frac{1}{2(1-\varepsilon)}}\left(\frac{ \delta}{2(1-\varepsilon)}e^{-\frac{1}{1-\varepsilon}}-\frac{ \delta^{3/2}}{6(1-\varepsilon)^2}e^{-\frac{7}{8(1-\varepsilon)}}\right).$$ There exists $\Delta_\theDeltaMainThm(\varepsilon,\delta) > 0$ such that if $G$ is a $\delta$-sparse graph of maximum degree $\Delta > \Delta_\theDeltaMainThm(\varepsilon,\delta)$, then $\chi_c(G) \leq (1-\varepsilon)\Delta$.
 \end{theorem}

\begin{proof}
For convenience we define 
\begin{equation}\label{eq:gFunc}
g(\varepsilon, \delta) = \frac{ \delta}{2(1-\varepsilon)}e^{-\frac{1}{1-\varepsilon}}-\frac{ \delta^{3/2}}{6(1-\varepsilon)^2}e^{-\frac{7}{8(1-\varepsilon)}}.
\end{equation}

Set $k = \lfloor(1-\varepsilon)\Delta\rfloor$ and let $C$ be a $k$-correspondence assignment for $G$. We will show that $G$ is $C$-colourable by repeatedly applying Lemma~\ref{lem:colouring} to the remaining uncoloured graph. We frequently assume that the maximum degree of this graph is sufficiently large, and explain at the end of this proof why we may do this.

To simplify the analysis, let $\varepsilon'>0$ be such that $(1-\varepsilon')\Delta = k$. If $\Delta_\theDeltaMainThm$ is large enough, then this can always be done in such a way that $\varepsilon'$ and $\delta$ still satisfy the conditions of the theorem provided $\Delta > \Delta_\theDeltaMainThm$. We also choose some $\delta' < \delta$ such that $\delta'$ and $\varepsilon'$ still satisfy the condition. When iterating the procedure, the sparsity of the uncoloured subgraph may decrease, but by taking $\Delta_\theDeltaMainThm$ to be large enough, we will ensure that it never drops below $\delta'$. In this way, the condition of the theorem is always satisfied and we can apply the naive colouring procedure again.

Let $\beta > 0$ be a constant such that $\varepsilon' e^{-\frac{1}{2(1-\varepsilon')}} + \beta < g(\varepsilon', \delta')$. Informally, we show that in the subgraph induced by uncoloured vertices, the ratio of number of colours available over maximum degree increases by at least $\beta/2$ after every iteration of the naive colouring procedure. Thus, this ratio eventually reaches $1$, which guarantees we can colour the final uncoloured subgraph greedily. Additionally, note that the upper-bound on the number of iterations does not depend on $\Delta$.

Let $T= \lceil\frac{2\varepsilon}{\beta}\rceil + 1$. First we define parameters for the small constants we use as follows. Define for each integer $i$ such that $0\le i \le T$ the following:
\begin{itemize}
    \item $\varepsilon_i = \varepsilon' - i \frac{\beta}{2}$
    \item $\gamma_i = \varepsilon_{i} e^{-\frac{1}{2(1-\varepsilon_{i})}} + \beta$
    \item $\delta_i = \delta - \frac{i}{T} (\delta - \delta')$
\end{itemize}

Let $r_0 = \Delta$. We now define the constants we use for the numbers of colours, degrees and quasirandomness as follows. Define for each integer $i$ such that $0\le i \le T$ the following:

\begin{itemize}
    \item $k_i = (1 - \varepsilon_i)r_i$
    \item $\mu_{i} = 1- (1-\frac{1}{2k_i})^{r_i}$
    \item $r_i = \left(\mu_{i-1}+\frac{\beta}{2}\right)r_{i-1}$ for $i\ne 0$
\end{itemize}

First we argue that $r_T$ will be large enough provided that $\Delta_\theDeltaMainThm(\varepsilon, \delta)$ is, as follows.

\begin{claim}\label{cl:DeltaLargeEnough}
For every $C_T$, there exists $C$ such that if $\Delta_\theDeltaMainThm(\varepsilon, \delta) > C$, then $r_T > C_T$.
\end{claim}
\begin{proof}
Note that given $\varepsilon$ and $\delta$, we have that $\mu_i \geq 1 - e^{\frac{1}{2(1-\varepsilon_i)}} > 0$ for every $i \in \{1, \dots, T\}$. Since $r_{i+1} \geq \mu_i r_i - \sqrt{r_i}(\log r_i)^5$, we have that $r_{i+1}$ grows with $r_i$ for every $i \in \{1, \dots, T\}$.
\end{proof}

We then argue two useful monotone properties.

\begin{claim}\label{cl:deltaeps}
If $\delta^* \geq \delta'$ and $\varepsilon^* \leq \varepsilon'$, then $g(\varepsilon^*,\delta^*)\geq g(\varepsilon',\delta')$.
\end{claim}

\begin{proofclaim}
It is easily checked that for fixed $\varepsilon$, the function $g$ is increasing in $\delta$. Therefore, it remains to argue that $g(\varepsilon^*,\delta')\geq g(\varepsilon',\delta')$, in other words, that $g$ is a decreasing function of $\varepsilon$ for fixed $\delta$.
We point out that $x \mapsto \frac{\delta'}{2x}e^{-\frac{1}{x}}$ is an increasing function of $x$ for $x \in [0.5,1]$, as well as $x \mapsto -\frac{\delta'^{\frac32}}{6x}e^{-\frac{7}{8x}}$ for $x \in [0.5,1]$. It follows that the sum is also an increasing function of $x\in [0.5,1]$. Setting $x=1-\varepsilon$, we obtain that $g$ is a decreasing function of $\varepsilon$ for $\varepsilon \in ]0,0.5]$ and for fixed $\delta$.
\end{proofclaim}

We can similarly argue the following.

\begin{claim}\label{cl:decreasinggamma}
For every $i \in \{1,\ldots,T-1\}$, we have $\gamma_{i+1}\leq \gamma_i$.
\end{claim}

\begin{proofclaim}
Since $\gamma_i = \varepsilon_{i} e^{-\frac{1}{2(1-\varepsilon_{i})}} + \beta$ and $\varepsilon_{i}$ is a decreasing function of $i$, it suffices to argue that the function $x \mapsto x \cdot e^{-\frac{1}{2(1-x)}}$ is increasing on $[0,0.5]$. This is easy to check by computing its derivative ($x \mapsto (1-\frac{x}{2(1-x^2)}) \cdot e^{-\frac{1}{2(1-x)}}$) and noticing that it is positive on $[0,0.5[$.
\end{proofclaim}

Now we argue that there inductively exists by Lemma~\ref{lem:colouring} subgraphs of $G$ and new correspondence assignments for those subgraphs whose parameters are defined as above.

\begin{claim}\label{cl:iterate}
There exist a family of graphs $(G_i: i\in [0,T])$ with $G_0=G$ and correspondence assignments $C_i$ for $G_i$ with $C_0=C$ such that all of the following hold for each $i\in [1,T]$:
\begin{enumerate}[1)]
    \item $G_i$ is $\delta_i$-sparse \label{item:delta_i}
    \item $G_{i}$ has maximum degree at most $r_i$ \label{item:Delta_i}
    \item $C_i$ is a $k_i$-correspondence assignment \label{item:colours}
    \item If there exists a $C_{i}$-colouring of $G_{i}$, then there exists a $C$-colouring of $G$\label{item:extend}
\end{enumerate}
\end{claim}
\begin{proof}
We proceed by induction on $i$. Hence $G_{i-1}$ is $\delta_{i-1}$-sparse graph with maximum degree $r_{i-1}$ and $C_{i-1}$ is $k_{i-1}$-correspondence assignment for $G_{i-1}$ such that~\ref{item:extend}) holds.

By choice of $\beta$, we have $\varepsilon'\cdot e^{-\frac{1}{2(1-\varepsilon')}}+\beta< g(\varepsilon',\delta')$. Since the left term is exactly $\gamma_0$, we rewrite the previous equation: $\gamma_0< g(\varepsilon',\delta')$.

Note that for every $i$, we have $\varepsilon_i \leq \varepsilon'$ and $\delta_i \geq \delta'$. By combining Claims~\ref{cl:deltaeps} and~\ref{cl:decreasinggamma}, we obtain,
for each $i \geq 1$, 
$$\gamma_i = \varepsilon_i e^{-\frac{1}{2(1-\varepsilon_i)}} + \beta < g(\varepsilon_i, \delta_i).$$ 

By the remark following Proposition~\ref{prop:partial}, there exists a $r_{i-1}$-regular graph $G'_{i-1}$ of sparsity $\delta_{i-1}$ that contains $G_{i-1}$ as a subgraph. We extend $C_{i-1}$ to $k_{i-1}$-correspondence assignment of $G$ arbitrarily.

Applying Lemma~\ref{lem:colouring} with $\gamma=\gamma_{i-1},G=G'_{i-1}, \Delta = r_{i-1}, \delta=\delta_{i-1}, k=k_{i-1}$ and $C=C_{i-1}$, we find that there exists a $\mu_{i-1}$-quasirandom subgraph of $G'_{i-1}$, call it $G_i$, such that there exists an $\ell$-correspondence assignment $C_{i}$ of $G_i$ such that any $C_i$-coloring of $G_i$ extends to $C_{i-1}$-coloring of $G'_{i-1}$. Hence there exists a $C_{i-1}$-coloring of $G_{i-1}$ and $\ell \ge k_{i-1} - ( 1-\mu_{i-1}-\gamma_{i-1})r_{i-1}$. 

Since $r_i$ is large enough by Claim~\ref{cl:DeltaLargeEnough}, Lemma~\ref{lem:stablesparsity} implies that $G_i$ is $\delta_i$-sparse and hence~\ref{item:delta_i}) holds for $G_i$. Since~\ref{item:extend}) holds for $G_{i-1}$, we find that~\ref{item:extend}) holds for $G_i$.

\begin{subclaim}
\ref{item:Delta_i}) holds for $G_i$.
\end{subclaim}
\begin{proof}
 Since $G_{i}$ is a $\mu_{i-1}$-quasirandom subgraph of $G_{i-1}$, it follows that 
$$r_i \leq \mu_{i-1}r_{i-1} +  \sqrt{r_{i-1}}(\log r_{i-1})^5.$$

\noindent Since $r_i$ is large enough, we have that $\sqrt{r_{i-1}}(\log r_{i-1})^5\le \frac{\beta}{2}r_{i-1}$ and hence

$$r_i \leq \left(\mu_{i-1}+\frac{\beta}{2}\right)r_{i-1},$$
as desired.
\end{proof}

\begin{subclaim}
\ref{item:colours}) holds for $G_i$.
\end{subclaim}
\begin{proofclaim}
It suffices to show that $\ell \ge k_i$. Recall that $\ell \ge k_{i-1} - (1-\mu_{i-1}-\gamma_{i-1})r_{i-1}$. 

Since $k_{i-1} = (1- \varepsilon_{i-1})r_{i-1}$ and $\gamma_{i-1} = \varepsilon_{i-1} e^{-\frac{1}{2(1-\varepsilon_{i-1})}} + \beta$, we have that
\begin{align*}
k_{i} &\geq (1-\varepsilon_{i-1})r_{i-1}-(1-\mu_{i-1})r_{i-1} + (\varepsilon_{i-1} e^{-\frac{1}{2(1-\varepsilon_{i-1})}} + \beta)r_{i-1}\\
& = \left(\mu_{i-1} - (1-e^{-\frac{1}{2(1-\varepsilon_{i-1})}})\varepsilon_{i-1} + \beta  \right)r_{i-1}.
\end{align*}
Since $\mu_{i-1} \geq 1 - e^{-\frac{1}{2(1-\varepsilon_{i-1})}}$, we find that
\begin{align*}
k_{i} &\geq ((1 - \varepsilon_{i-1})\mu_{i-1} + \beta) r_{i-1}\\
&= (1 - \varepsilon_{i-1})\mu_{i-1}r_{i-1} + \beta r_{i-1}.
\end{align*}

Since $r_i = \left(\mu_{i-1}+\frac{\beta}{2}\right)r_{i-1}$, we have that
\begin{align*}
k_{i} &= (1 - \varepsilon_{i-1})(r_i - \frac{\beta}{2}r_{i-1}) + \beta r_{i-1}\\
&= (1 - \varepsilon_{i-1})r_i + \frac{\beta}{2}r_{i-1}\\
&\geq \left(1 - \varepsilon_{i-1} + \frac{\beta}{2}\right)r_{i}\\
&= (1-\varepsilon_i)r_i,
\end{align*}
as desired.
\end{proofclaim}

\end{proof}

Since $T= \lceil\frac{2\varepsilon}{\beta}\rceil + 1$ and $\varepsilon_T = \varepsilon - T \frac{\beta}{2}$, we find that $\varepsilon_T < 0$. Hence $k_T = (1-\varepsilon_T)r_T > r_T + 1$ provided $r_T$ is large enough. Thus there exists a $C_T$-colouring of $G_T$ using a greedy algorithm. By Claim~\ref{cl:iterate}(\ref{item:extend}), it follows that $G$ is $C$-colourable.
\end{proof}

Bruhn and Joos~\cite{BruhnJoos} note that for $\delta \in [0, 0.9]$, setting $\varepsilon = 0.1827\delta - 0.0778\delta^{3/2}$ satisfies $\varepsilon < g(\varepsilon, \delta)$, where $g$ is the function defined in~\eqref{eq:gFunc}. Since $\sqrt{e} < e^{\frac{1}{2(1-\varepsilon)}}$ for all $\varepsilon>0$, we have that setting $\varepsilon = (0.1827\delta - 0.0778\delta^{3/2})\sqrt{e}$ satisfies $\varepsilon < e^{\frac{1}{2(1-\varepsilon)}}g(\varepsilon, \delta)$ for $\delta$ in the same range. Therefore we deduce Theorem~\ref{thm:SparsityApprox} as a corollary.

\section{Application to Strong Edge Colouring}\label{sect:strongedge}
 
In this section we prove Theorem~\ref{thm:MainThmStrongEdge}. Recall that $L(H)$ denotes the \emph{line graph} of $H$, that is, the graph with vertex set $E(H)$ and where two edges are adjacent if they were incident in $H$. The \emph{square} of a graph $G$ is obtained from $G$ by adding an edge between every pair of vertices $u, v \in V(G)$ which have distance precisely $2$ in $G$. In other words, two vertices are adjacent in the square of $G$ if and only if they are at distance $1$ or $2$ in $G$. If $H$ is a graph, we denote the square of the line graph of $H$ by $L^2(H)$. Thus, a strong edge colouring of $H$ is a vertex colouring of $L^2(H)$. If $uv \in E(H)$, then $N^s (uv)$ denotes the \emph{strong neighborhood} of $uv$, \textit{i.e.} the set of edges of $H$ which have an endpoint adjacent to $u$ or $v$. Equivalently, $N^s(uv)$ is the neighbourhood of the vertex $uv$ in the graph $L^2(H)$. We also let $d^s(uv) = |N^s (uv)|$. Given a set of vertices $A$ and a vertex $u$, we define $d_{\overbar{A}}(u)$ as $d(u) - d_A(u)$. Similarly, given a set of edges $B$ and an edge $uv$, we define $d^s_{\overbar{B}}(uv)$ as $|N^s (uv)\setminus B|$.

\subsection{A Sparsity Bound for Squares of Linegraphs}
Molloy and Reed~\cite{MolloyReedSCI} and Bruhn and Joos~\cite{BruhnJoos} gave evidence for Conjecture~\ref{conj:1.25} by improving the constant from the trivial bound of $2\Delta^2$. To do this they showed that for any graph $H$, the graph $L^2(H)$ is a subgraph of a graph $G$ such that $\Delta(G)=2\Delta(H)^2$ and $G$ is $\delta$-sparse for some $\delta >0$. This was achieved directly by carefully bounding the number of edges induced by the strong neighbourhood of an edge of $H$. Bruhn and Joos obtained the following inequalities and bounds which we shall make use of later. 
  
  \begin{lemma}\emph{\cite{BruhnJoos}}\label{lem:BruhnJoosBounds}
 Let $H$ be a graph of maximum degree $\Delta$, and $G = L^2(H)$. Let $uv \in E(H)$ and define $X = N_H(u)\cup N_H(v) \setminus \{u,v\}$ and $Y = N_H(X) \setminus (X \cup \{u,v\})$. Letting $C_4(X,Y)$ denote the number of $4$-cycles $x_1y_1x_2y_2$ such that $x_1,x_2 \in X$ and $y_1, y_2 \in Y$, we have $$\sd(uv) \leq (2-\alpha - \beta) \Delta^2 - 2\Delta,$$
$$C_4(X,Y) \geq \frac12 \left( \frac{(2-\alpha - 2\beta  - \gamma)^2}{2(2-\alpha)^2}\Delta^4 - (7 - \frac{\gamma}2)\Delta^3 \right),$$ and $$|E(G[\sn(uv)])| \leq \left( 2- \alpha - \beta - \frac{\gamma}2 \right) \Delta^4 - 2 C_4(X,Y) + \left( \frac{\gamma}2 - 2 \right) \Delta^3,$$ where $\alpha \Delta = |N(u) \cap N(v)|$, $\beta \Delta^2 = |E(H[X])|$ and $\gamma \Delta^3 = \sum_{y \in Y}d_X(y)(\Delta - d_X(y))$.
 \end{lemma}
 We first slightly improve the bound on the number of edges induced by the strong neighbourhood of an edge.
 
 \begin{lemma}\label{lem:improvedBound}
 Let $H$ be a graph of maximum degree $\Delta$, and $G = L^2(H)$. Let $X = N_H(u)\cup N_H(v) \setminus \{u,v\}$ and $Y = N_H(X) \setminus (X \cup \{u,v\})$, and let $C_4(X,Y)$ denote the number of $4$-cycles $x_1y_1x_2y_2$ such that $x_1,x_2 \in X$ and $y_1, y_2 \in Y$. If $uv \in E(H)$, then $$|E(G[\sn(uv)])| \leq \left( 2- \alpha - \beta - \frac{\gamma}2 \right) \Delta^4 - 2 C_4(X,Y)  - \frac{\gamma^2}{2(2-\alpha-\beta)}\Delta^4 + \left( \frac{\gamma}2 - 2 \right) \Delta^3,$$ where $\alpha \Delta = |N(u) \cap N(v)|$, $\beta \Delta^2 = |E(H[X])|$ and $\gamma \Delta^3 = \sum_{y \in Y}d_X(y)(\Delta - d_X(y))$.
 \end{lemma}
 
 \begin{proof}
 By the remark following Proposition~\ref{prop:partial}, we may assume that $H$ is $\Delta$-regular. 
 Let $Z = \sn(uv)$. 
 
 We denote by $P$ the number of all (directed) paths $pqrs$ such that $pq \in Z$. We denote by $B$ the number of all paths $pqrs$ such that $pq \in Z$ and $r, s \not\in X$.

 \begin{claim}\label{clm:PB4C4}
 We have $2|E(G[\sn(uv)])| \leq P - B - 4 |C_4(X,Y)|$.
 \end{claim}
 
 \begin{proofclaim}
  We note that $2|E(G[\sn(uv)])|$ is at most the number of paths $pqrs$ where $pqrs$ is a path with $pq\in Z$ and $rs \in Z$. Since every edge in $Z$ has an endpoint in $X$, this is at most $P-B$.
  
  In fact, if $pqrs$ is a cycle, then we count both paths $pqrs$ and $qpsr$ for the edge $(pq,rs)$. If $pqrs$ is a cycle in $C_4(X,Y)$, then this double couting is repeated for each directed pair of opposite edges on the cycle. We derive $2|E(G[\sn(uv)])| \leq P - B - 4 |C_4(X,Y)|$, as desired.
 \end{proofclaim}
 
 We note that $P \leq 2 \Delta^2 \cdot |Z|$. From Lemma~\ref{lem:BruhnJoosBounds}, we know that $|Z|\leq (2-\alpha - \beta) \Delta^2 - 2\Delta$, hence $P \leq 2 \Delta^2 \cdot ((2-\alpha - \beta) \Delta^2 - 2\Delta)$. We focus on lower-bounding $B$. 
 
 We will lower-bound $B$ by considering two distinct types of such paths, as follows.
We denote by $B_1$ the number of all paths $(p,x,y,q)$ such that $x \in X$, $y\in Y$ and $q \not\in X$. We denote by $B_2$ the number of all paths $(x,y,w,z)$ such that $x \in X$, $y \in Y$ and $w, z \not\in X$. See Figure~\ref{fig:EXY} for an illustration of both types.
 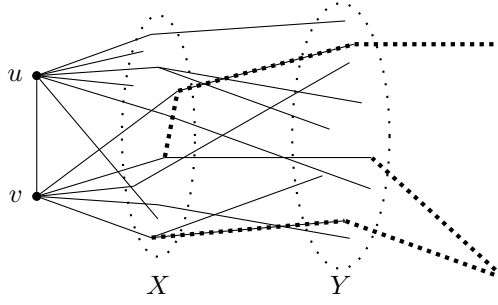
\begin{figure}[h]
\centering
\begin{tikzpicture}[scale=1.6]
\tikzstyle{whitenode}=[draw=white,circle,fill=white,minimum size=0pt,inner sep=0pt]
\tikzstyle{blacknode}=[draw,circle,fill=black,minimum size=3pt,inner sep=0pt]
\draw (0,0) node[blacknode] (u) [label=left:$u$] {};
\draw (u)
--++ (-90:1cm)  node[blacknode] (v) [label=left:$v$] {};
\draw[thick,loosely dotted] (1,-0.5) ellipse[x radius=0.3cm, y radius=1cm];
\draw (1,-1.5) node (e1) [label=-90:$X$] {};
\draw[thick,loosely dotted] (2.5,-0.5) ellipse [x radius=0.4cm, y radius=1.1cm];
\draw (2.5,-1.5) node (e2) [label=-90:$Y$] {};
\foreach \i/\j in {20/1,13/0.9,4/1,-6/0.8,-17/1.1,-50/1.55} {\draw (u) --++ (\i:\j cm) node[whitenode] (a\i) {};}
\foreach \i/\j in {20/1,4/1,-6/0.8,-17/1.1,-37/1.45} {\draw (v) --++ (-\i:\j cm) node[whitenode] (b\i) {};}

\foreach \i/\j/\k/\l in {20/1/20/1.5,20/1/5/1.6,4/1/-10/1.6,-6/0.8/30/2.05,-17/1.1/0/1.7,-37/1.45/15/1.5} {\draw (v) ++ (-\i:\j cm) --++ (\k:\l cm) node[whitenode] (c\i) {};}
\foreach \i/\j/\k/\l in {20/1/4/1.6,4/1/-10/1.7,4/1/-20/1.5,-17/1.1/-20/1.8} {\draw (u) ++ (\i:\j cm) --++ (\k:\l cm) node[whitenode] (d\i) {};}

\draw[ultra thick,dotted] (b-17) -- (b-37) -- (c-37) --++ (0:1.2cm) node (x) {};

\draw[ultra thick,dotted] (b20) -- (c20) --++ (-20:1.4cm) node (x) {} -- (c-17);

\end{tikzpicture}\caption{An example of a path $pxyq$ (dotted, top) and of a path $xywz$ (dotted, bottom).}\label{fig:EXY}
\end{figure}

Note that both types are indeed taken into account in $B$, and that no path can be of both types: the second vertex belongs to $X$ in the case of $B_1$, to $Y$ in the case of $B_2$. Therefore, we have $B \geq B_1+B_2$.

 \begin{claim}\label{clm:B1}
 We have $B_1 \geq  \gamma \Delta^4 - \gamma \Delta^3$.
 \end{claim}
 
 \begin{proofclaim}
 We prove this claim following Bruhn and Joos~\cite[Lemma 2.1]{BruhnJoos}. Since $H$ is $\Delta$-regular, for every fixed path $xyq$ with $x \in X$, $y \in Y$ and $q \not\in X$, there are $\Delta-1$ choices of $p$ to extend it. The number of such $xyq$ is $$\sum_{y\in Y}d_X(y)(\Delta-d_X(y))=\gamma \Delta^3.$$
 
 It follows that $B_1 \geq \gamma \Delta^3 \cdot (\Delta-1)$, hence the conclusion.

 \end{proofclaim}
 
  \begin{claim}\label{clm:B2}
 We have $B_2 \geq \frac{\gamma^2}{(2-\alpha-\beta)} \Delta^4$.
 \end{claim}
 
 \begin{proofclaim}
  Each vertex $y \in Y$ can be extended to a path in $R$ by choosing a neighbour of $y$ in $X$, and a path of length two starting at $y$ and avoiding $X$. Thus, since $H$ is
$\Delta$-regular, we have
 \begin{align*}
|B_2| &= \sum_{y \in Y} \left( d_X(y) \sum_{w \in N(y) \setminus X}d_{\overbar{X}}(w) \right)\\
& = \sum_{y \in Y} \left( d_X(y)\sum_{w \in N(y) \setminus X}(\Delta - d_{X}(w)) \right). \numberthis\label{eq:ExtraPaths1}
 \end{align*}
 Expanding the sum, equation~\eqref{eq:ExtraPaths1} becomes
 \begin{equation}\label{eq:ExtraPaths2}
|B_2| = \sum_{y \in Y} d_X(y)\Delta d_{\overbar{X}}(y) - \sum_{y \in Y} \sum_{w \in N(y) \setminus X} d_X(y) d_{X}(w).
 \end{equation}
If, for some $y \in Y$, $w$ is adjacent to $y$ and not in $X$, then either $w \in Y$, or $d_X(w) = 0$. Thus, the second sum in~\eqref{eq:ExtraPaths2} is really a sum over the edges of $H[Y]$.
\begin{align*}
\sum_{y \in Y} \sum_{w \in N(y) \setminus X} d_X(y) d_{X}(w) & = \sum_{y \in Y} \sum_{w \in N(y) \cap Y}d_X(y) d_{X}(w)\\
&=\sum_{yw \in E(H[Y])} 2d_X(y) d_{X}(w).
\end{align*}
Since $2ab \leq a^2+b^2$ for all integers $a$ and $b$, we have 
\begin{align*}
\sum_{yw \in E(H[Y])} 2d_X(y) d_{X}(w) & \leq \sum_{yw \in E(H[Y])} (d_X(y)^2 + d_{X}(w)^2)\\
& = \sum_{y \in Y}d_Y(y)d_{X}(y)^2\\
& \leq \sum_{y \in Y}d_{\overbar{X}}(y)d_{X}(y)^2.\numberthis\label{eq:ExtraPaths3}
\end{align*}
Substituting the expression in~\eqref{eq:ExtraPaths3} into~\eqref{eq:ExtraPaths2}, recombining and simplifying gives
\begin{align*}
|B_2| &\geq \sum_{y \in Y} \left( d_X(y)\Delta d_{\overbar{X}}(y) - d_{\overbar{X}}(y)d_{X}(y)^2 \right)\\
&= \sum_{y \in Y}  d_X(y)d_{\overbar{X}}(y)(\Delta - d_{X}(y)) \\
&= \sum_{y \in Y} d_X(y)d_{\overbar{X}}(y)^2.\numberthis \label{eq:ExtraPaths4}
\end{align*}
Let us denote by $E(X,Y)$ the set of edges with an endpoint in $X$ and the other in $Y$. If $e \in E(X,Y)$, we denote by $y_e$ the endpoint of $e$ in $Y$. Writing the sum in~\eqref{eq:ExtraPaths4} as a sum over edges we obtain $$|B_2| \geq \sum_{y \in Y} d_X(y)d_{\overbar{X}}(y)^2 =\sum_{e \in E(X,Y)} d_{\overbar{X}}(y_e)^2.$$ 

Now using the Cauchy-Schwarz inequality, we have
\begin{align*}
|B_2| &\geq |E(X,Y)| \left( \frac{\sum_{e \in E(X,Y)} d_{\overbar{X}}(y_e)}{|E(X,Y)|}\right)^2\\
&=\frac{1}{|E(X,Y)|}\left( \sum_{e \in E(X,Y)} d_{\overbar{X}}(y_e)\right)^2\\
 &=  \frac{1}{|E(X,Y)|}\left( \gamma \Delta^3 \right)^2. \numberthis \label{eq:ExtraPathsFinal}
\end{align*}
By Lemma~\ref{lem:BruhnJoosBounds}, we have that $|E(X,Y)| \leq \sd(uv) \leq (2-\alpha - \beta)\Delta^2$.  Substituting this into~\eqref{eq:ExtraPathsFinal} gives that $|B_2| \geq \frac{\gamma^2}{(2-\alpha-\beta)}\Delta^4$ as claimed.
 \end{proofclaim}
 
 Combining Claims~\ref{clm:PB4C4},~\ref{clm:B1} and~\ref{clm:B2} with the fact that $P \leq 2 \Delta^2 \cdot ((2-\alpha - \beta) \Delta^2 - 2\Delta)$, we obtain the desired bound.
 
 \end{proof}

\subsection{Restricting the Set of Interesting Edges}

Let $G$ be a graph with maximum degree $r$ such that for every vertex $u \in V(G)$, the graph induced by the neighbourhood of $u$ has at most $(1-\delta){r \choose 2}$ edges. Theorem~\ref{thm:SparsityApprox} shows that there is some $\gamma > 0$, which increases with the sparsity, such that $G$ is colourable with $(1-\gamma)r$ colours. However given this fact, one need not colour all the vertices. Indeed if $A \subseteq V(G)$ is the set of vertices with degree at least $(1-\gamma)r$, then it suffices to colour $A$. After this, the remaining vertices of $G$ can be coloured greedily without introducing any new colours. In fact, we can repeat this argument to show that we only need to colour the maximum subgraph $F$ of $G$ with minimum degree at least $(1-\gamma)r$ (note that $F$ may be empty). We show that in our application to the strong chromatic index, the graph $F$ thus obtained is even sparser than $G$.

\begin{lemma}\label{lem:Fdensity}
Let $H$ be a graph with maximum degree $\Delta$, and set $G = L^2(H)$. Let $\eta \in [0,0.3]$ be a fixed constant and let $F \subseteq E(H)$ be the maximum set of edges $e$ such that $d_F(e) \geq (2-\eta)\Delta^2$. Finally, for $e\in E(H)$, let $F_{e}$ be the set $F \cap N_G(e)$. If $e \in E(F)$, then $$|E(G[F_{e}])| \leq \left( \frac{31}{6} - \frac{128}{3(10 -3\eta)}+4\eta -\eta^2 \right) \Delta^4.$$
\end{lemma}

\begin{proof}
Let $e$ be an edge $uv$ of $H$ such that $e \in F$. Let $X = N_H(u)\cup N_H(v) \setminus \{u,v\}$ and $Y = N_H(X) \setminus (X \cup \{u,v\})$. We define an auxiliary graph $C_4(e)$ whose vertex set is $E(X,Y)$, and whose edges consist of those pairs $\{f_1,f_2\} \subseteq E(X,Y)$ such that $f_1$ and $f_2$ are opposite edges of a $4$-cycle in $C_4(X,Y)$. For an edge $f \in E(H)$, we have $d_G(f) \leq 2\Delta^2 - d_{C_4(e)}(f)$. If $f$ belongs to $F$, by definition of $F$ we have $(2-\eta)\Delta^2 \leq d_G(f)$. Therefore, for any edge $f \in F$, we have $d_{C_4(e)}(f) \leq \eta \Delta^2$. Note also that $$\sum_{g \in E(X,Y) \setminus F}d_{\sn(e)}(g) \geq \sum_{g \in E(X,Y)\setminus F}d_{C_4(e)}(g)$$ and $$4 C_4(X,Y) = \sum_{f \in E(X,Y)\cap F}d_{C_4(e)}(f) + \sum_{g \in E(X,Y) \setminus F}d_{C_4(e)}(g).$$
Combining these observations we have $$\sum_{g \in E(X,Y) \setminus F}d_{\sn(e)}(g) \geq 4 C_4(X,Y) - \eta |F| \Delta^2.$$
Finally, 
\begin{align*}
|E(G[F_{e}])| &\leq |E(G[\sn(e)])| - \sum_{g \in E(X,Y)\setminus F} d_{\sn(e)}(g) + |E(G[E(X,Y)\setminus F])|\\ 
&\leq |E(G[N^s(e)])| -  4 C_4(X,Y) + \eta |F| \Delta^2 + \frac{1}{2}|E(X,Y)\setminus F|^2. \numberthis \label{eq:newBound}
\end{align*}

Note that $|E(X,Y) \setminus F| \leq |N_G(e) \setminus F|$. Since $e \in F$, by definition of $F$ we have $|N_G(e) \cap F|\geq (2-\eta)\Delta^2$. Thus, using Lemma~\ref{lem:BruhnJoosBounds}, we have $$|E(X,Y) \setminus F| \leq (2-\alpha-\beta) \Delta^2 - (2-\eta)\Delta^2 = (\eta-\alpha-\beta)\Delta^2,$$
where, as usual, $\alpha\Delta = |N(u) \cap N(v)|$ and $\beta\Delta^2 = |E(G[X])|$. We can now bound the last two terms in equation~\eqref{eq:newBound}. 
\begin{align*}
\eta |F| \Delta^2 + \frac{1}{2}|E(X,Y)\setminus F|^2 &\leq \eta |F| \Delta^2 + \frac{1}{2}(\eta-\alpha-\beta)|E(X,Y)\setminus F|\Delta^2\\
&\leq \eta  ( |F| + |E(X,Y)\setminus F|) \Delta^2\\
&\leq \eta  |\sn(e)|\Delta^2\\
&\leq \eta (2-\alpha-\beta)\Delta^4.
\end{align*}
Therefore, by Lemma~\ref{lem:improvedBound} and the above, inequality~\eqref{eq:newBound} becomes $$|E(G[F_{e}])| \leq \left( 2- \alpha - \beta - \frac{\gamma}2 \right) \Delta^4 - \frac{\gamma^2}{2(2-\alpha-\beta)}\Delta^4 + \left( \frac{\gamma}2 - 2 \right) \Delta^3 - 6 C_4(X,Y) + \eta (2-\alpha-\beta)\Delta^4$$
Now, using the expression for $C_4(X,Y)$ from Lemma~\ref{lem:BruhnJoosBounds} gives

\begin{equation}\label{eq:finalBound}
|E(G[F_e])| \leq f(\alpha, \beta, \gamma, \eta)  \Delta^4 + (19-\gamma) \Delta^3,
\end{equation}
where 
\begin{equation}\label{eq:fBound}
f(\alpha, \beta, \gamma, \eta) = 2- \alpha - \beta - \frac{\gamma}2  - \frac{3(2-\alpha - 2\beta - \gamma)^2}{2(2-\alpha)^2}  - \frac{\gamma^2}{2(2-\alpha-\beta)}
 + \eta (2-\alpha-\beta).
\end{equation}

It remains to show that $f(\alpha, \beta, \gamma, \eta) \leq \frac{9}{10} + \eta\left( 4 - \eta - \frac{64}{5(10-3\eta)}\right)$. 
By Lemma~\ref{lem:BruhnJoosBounds} and the definition of $F$, we have $(2-\eta)\Delta^2 \leq d_G(e) \leq (2-\alpha-\beta)\Delta^2$. Thus $\alpha+\beta \leq \eta$. Letting $x = \beta + \frac{\gamma}{2}$, equation~\eqref{eq:fBound} simplifies to
\begin{equation}\label{eq:fBoundx}
f(\alpha, \beta, \gamma, \eta) = f_0(\alpha, \beta, \eta, x) =  2- \alpha - x  - \frac{3(2-\alpha - 2x)^2}{2(2-\alpha)^2}  - \frac{2(x-\beta)^2}{2-\alpha-\beta}
 + \eta (2-\alpha-\beta).
\end{equation}
We first investigate the dependence on $\alpha$. First, note that 
\begin{align*}
\frac{\partial f_0}{\partial \alpha} & = -1 - \frac32 \cdot 2 \cdot (- \frac{2x}{(\alpha-2)^2}) \cdot (1 - \frac{2x}{2-\alpha}) - \frac{2(x-\beta)^2}{(2-\alpha-\beta)^2} - \eta\\
 & = -1 + \frac{6x}{(2-\alpha)^2}\cdot (1 - \frac{2x}{2-\alpha}) - \frac{2(x-\beta)^2}{(2-\alpha-\beta)^2} - \eta.
\end{align*}
For any positive $a$, the function $x \mapsto x \cdot (1-a\cdot x)$ reaches a maximum of $\frac{1}{4a}$ at $x=\frac1{2a}$. Therefore, for all $\alpha$ in the range considered, the term $$\frac{6x}{(2-\alpha)^2} \left( 1 - \frac{2x}{2-\alpha} \right)$$
attains its maximum at $x = \frac{2-\alpha}{4}$. As $\alpha \leq 1$ by definition, we have $$\frac{6x(2-\alpha-2x)}{(2-\alpha)^3} \leq \frac{6\cdot(2-\alpha)\cdot \frac{2-\alpha}8}{(2-\alpha)^3} =\frac{3}{4(2-\alpha)} < 1,$$
 whence $\frac{\partial f}{\partial \alpha} <0$. Thus, defining $f_1(\beta, \eta, x) = f_0(0, \beta, \eta, x)$ we have
 \begin{equation}\label{eq:fboundxAlpha=0}
 f_0(\alpha, \beta, \eta, x) \leq f_1(\beta, \eta, x) = \frac{1}{2} + 2x - \frac{3}{2}x^2 - \frac{2(x-\beta)^2}{2-\beta}
 + \eta (2-\beta).
 \end{equation}
 For fixed $\beta$ and $\eta$, we calculate $$\frac{\partial f_1}{\partial x} = 2-3x-\frac{4(x-\beta)}{(2-\beta)},$$ so one can check that $\frac{\partial f_1}{\partial x}=0$ only when $x = \frac{4+2\beta}{10-3\beta}$. The second derivative $\frac{\partial^2 f_1}{\partial x^2}$ is easily seen to be negative, so $f_1$ attains its maximum at $x=\frac{4+2\beta}{10-3\beta}$. Thus we have $f_1(\beta, \eta, x) \leq f_2(\beta, \eta)$, where $$f_2(\beta, \eta) = f_1\left(\beta, \eta, \frac{4+2\beta}{10-3\beta} \right) = (2-\eta)\beta + \frac{31}{6} - \frac{128}{3(10 -3\beta)} + 2\eta.$$
Now, $$\frac{\partial f_2}{\partial \beta} = 2-\eta - \frac{128}{(10-3\beta)^2}.$$ Since $\beta \leq \eta$, we have $2-\eta - \frac{128}{(10-3\beta)^2} \geq 2-\eta - \frac{128}{(10-3\eta)^2}$ which is positive for $\eta \leq 0.3$. Thus $f_2$ is increasing in $\beta$. Again, since $\beta \leq \eta$, we conclude that $$f_2(\beta, \eta) \leq f_2(\eta, \eta) = 4\eta -\eta^2 + \frac{31}{6} - \frac{128}{3(10 -3\eta)}.$$
\end{proof}
 
This refined sparsity bound combined with our new colouring procedure is enough to prove Theorem~\ref{thm:MainThmStrongEdge}.

\begin{proof}[Proof of Theorem~\ref{thm:MainThmStrongEdge}]
Let $H$ be a graph of sufficiently large maximum degree, and $G = L^2(H)$. Let $\eta = 0.164$ and let $F$ be the set of edges described in Lemma~\ref{lem:Fdensity}. By the argument preceeding Lemma~\ref{lem:Fdensity}, it suffices to colour $G[F]$. By Lemma~\ref{lem:Fdensity}, for each edge $e \in E(H)$ we have $$|E(G[F_e])| \leq \left(4\eta -\eta^2 + \frac{31}{6} - \frac{128}{3(10 -3\eta)}\right)\Delta^4 + (19-\gamma)\Delta^3 < 1.309 \Delta^4,$$ provided $\Delta$ is large enough. Thus, $|E(G[F_e])| \leq (1-\delta){2\Delta^2 \choose 2}$, where $\delta = 0.345$. Note that $\delta = 0.345$ and $\varepsilon=0.0825$ satisfy the conditions of Theorem~\ref{thm:MainThmCorr}, so $G[F]$ is $(1-\varepsilon)2\Delta^2$-colourable. We derive that $H$ admits a strong edge colouring with at most $1.835 \Delta^2$ colours.
\end{proof}

\section{Reed's Conjecture}\label{sec:ReedsConj}

In this section we prove Theorem~\ref{thm:MainThmReeds}, by combining Theorem~\ref{thm:density} and Theorem~\ref{thm:MainThmCorr} with the technique of King and Reed~\cite{KingReedShort}. The key idea in King and Reed~\cite{KingReedShort} is that for any $\varepsilon > 0$, a smallest counterexample to Theorem~\ref{thm:KingReedEpsilon} cannot contain an independent set $S$ which hits every maximal clique. Otherwise one can check that deleting a maximal independent set containing $S$ produces a smaller counterexample. Thus, by the following result, we may deduce that a smallest counterexample has small clique number.

\begin{theorem}\label{thm:King}\emph{\cite{KingIndep}}
Every graph satisfying $\omega(G) > \frac23 (\Delta(G)+1)$ contains an independent set hitting every maximum clique.
\end{theorem}

Using Theorem~\ref{thm:density} and Theorem~\ref{thm:SparsityApprox}, we deduce a bound on the chromatic number of these graphs.

\newcounter{DeltaEpsilon}
\setcounter{DeltaEpsilon}{\thei}
\addtocounter{i}{1}
\begin{lemma}\label{lem:epsilon(alpha,delta)}
Let $G$ be a graph of maximum degree $\Delta$, and clique number $\omega = (1 - \alpha)(\Delta +1)$. There exists $\Delta_\theDeltaEpsilon(\varepsilon,\alpha)$ such that if $\Delta > \Delta_\theDeltaEpsilon(\varepsilon,\alpha)$, then $\chi(G) \leq \lceil (1-\varepsilon)(\Delta+1) + \varepsilon \omega \rceil $, provided $\varepsilon \leq  0.3012 \frac{\alpha}2 (1 - 2\varepsilon)^2 - 0.1283 \frac{\alpha^2}{2\sqrt{2}} (1 - 2\varepsilon)^3$.
\end{lemma}

\begin{proof}
Let $k = \lceil (1-\varepsilon)(\Delta+1) + \varepsilon \omega \rceil$. We set $\varepsilon' = \varepsilon\alpha$ so that $k= \lceil (1 - \varepsilon')(\Delta+1) \rceil$. It suffices to show that $G$ is $k$-colourable. To do so, we may first assume that $G$ is a critical graph. Now by Theorem~\ref{thm:density}, we have that $G$ is $\delta$-sparse where $\delta = \frac12 ( \alpha - 2\varepsilon')^2$. By Theorem~\ref{thm:SparsityApprox}, such a graph can be coloured with $(1- \varepsilon')(\Delta+1)$ colours, provided $\varepsilon' \leq 0.3012\delta - 0.1283\delta^{3/2}$. This simplifies to $\varepsilon \leq  0.3012 \frac{\alpha}2 (1 - 2\varepsilon)^2 - 0.1283 \frac{\alpha^2}{2\sqrt{2}} (1 - 2\varepsilon)^3$, which is satisfied by assumption.
\end{proof}

\begin{figure}

\centering
\begin{tabular}{cccccccc}
\toprule
 $\alpha$ & $\varepsilon$ & &  $\alpha$ & $\varepsilon$  & & $\alpha$ & $\varepsilon$ \\
 \midrule
0.02 & 0.0029 & & 0.32 & 0.0375 & & 0.62 & 0.0603\\
0.04 & 0.0058 & & 0.34 & 0.0393 & & 0.64 & 0.0615\\
0.06 & 0.0085 & & 0.36 & 0.0411 & & 0.66 & 0.0627\\
0.08 & 0.0112 & & 0.38 & 0.0428 & & 0.68 & 0.0639\\
0.10 & 0.0138 & & 0.40 & 0.0445 & & 0.70 & 0.0651\\
0.12 & 0.0163 & & 0.42 & 0.0461 & & 0.72 & 0.0662\\
0.14 & 0.0187 & & 0.44 & 0.0477 & & 0.74 & 0.0673\\
0.16 & 0.0210 & & 0.46 & 0.0492 & & 0.76 & 0.0684\\
0.18 & 0.0233 & & 0.48 & 0.0507 & & 0.78 & 0.0694\\
0.20 & 0.0255 & & 0.50 & 0.0522 & & 0.80 & 0.0704\\
0.22 & 0.0277 & & 0.52 & 0.0536 & & 0.82 & 0.0715\\
0.24 & 0.0297 & & 0.54 & 0.0550 & & 0.84 & 0.0724\\
0.26 & 0.0318 & & 0.56 & 0.0564 & & 0.86 & 0.0734\\
0.28 & 0.0337 & & 0.58 & 0.0577 & & 0.88 & 0.0743\\
0.30 & 0.0356 & & 0.60 & 0.0590 & & 0.90 & 0.0752\\
 \bottomrule
\end{tabular}

\caption*{Table 1: Values of $\alpha$ and $\varepsilon$ which satisfy the conditions in Lemma~\ref{lem:epsilon(alpha,delta)}.}
\label{tableOmega}
\end{figure}

Since the condition in Lemma~\ref{lem:epsilon(alpha,delta)} is somewhat involved, Table 1 lists values of $\alpha$ and $\varepsilon$ which satisfy it. Before proving Theorem~\ref{thm:MainThmReeds}, we mention the following result which we require.

\newcounter{DeltaReed}
\setcounter{DeltaReed}{\thei}
\addtocounter{i}{1}
\begin{theorem}\emph{\cite{Reeds}}\label{th:reedo}
There is a constant $\Delta_\theDeltaReed$ such that any graph $G$ with $\Delta(G)\geq \Delta_\theDeltaReed$ and $\omega(G) \geq (1-\frac1{7\cdot 10^7}) \Delta(G)$ satisfies $\chi(G)\leq \frac{\Delta(G)+\omega(G)+1}2$.
\end{theorem}

We are now ready to prove Theorem~\ref{thm:MainThmReeds}.

\begin{proof}[Proof of Theorem~\ref{thm:MainThmReeds}]
Let $\Delta_\theDeltamain= \max (1.4\cdot 10^8 \cdot \Delta_\theDeltaSparsity, \Delta_\theDeltaReed)$ and let $G$ be a graph with $\Delta(G) > \Delta_\theDeltamain$. First note that if $\omega(G) > \Delta(G)-2\Delta_\theDeltaSparsity$, then by the choice of $\Delta_\theDeltamain$ we have $\omega(G) \geq (1-\frac1{7 \cdot 10^7})\cdot \Delta(G)$. By Theorem~\ref{th:reedo}, the conclusion strongly holds. Therefore, from now on we can assume that $\omega(G) \leq \Delta(G)-2\Delta_\theDeltaSparsity$.

If $G$ has clique number $\omega(G) \leq \frac{2}{3}(\Delta(G)+1)$, then we set $G'' = G$ in what follows. Otherwise, if $\omega(G) > \frac{2}{3}(\Delta(G)+1)$, then by Theorem~\ref{thm:King} there is an independent set $S \subseteq V(G)$ which contains a vertex of every clique of size $\omega(G)$. Extend $S$ to a maximal independent set $S'$ and set $G' = G- S'$. Since $S'$ is maximal, $\Delta(G') \leq \Delta(G)-1$, and since $S'$ contains a vertex in every maximal clique, $\omega(G') = \omega(G) -1$. Furthermore, $\chi(G')\geq \chi(G)-1$ since $S$ is an independent set. While $\omega(G') > \frac23 (\Delta(G') + 1)$, we repeatedly apply this reduction until we obtain a graph $G''$ with $\omega(G'') \leq \frac {2}{3}(\Delta(G'')+1)$. Since the clique number decreases by precisely one each time, the process terminates after $p$ steps, where $p \leq \omega(G)$. 

Note that $\Delta(G'') \leq \Delta(G)-p$, $\omega(G'') = \omega(G) -p$ and $\chi(G'')\geq \chi(G)-p$. If $\Delta(G'') > \Delta_\theDeltaSparsity$, then Lemma~\ref{lem:epsilon(alpha,delta)} implies that $\chi(G'') \leq \frac{25}{26}\Delta(G'') + \frac{1}{26}\omega(G'')$ provided $\frac{1}{26} \leq  0.3012 \frac{\alpha}2 (1 - \frac{2}{26})^2 - 0.1283 \frac{\alpha^2}{2\sqrt{2}} (1 - \frac{2}{26})^3$. This is easily seen to hold for all $1/3 \leq \alpha \leq 1$. We deduce that $\chi(G) \leq \frac{25}{26}\Delta(G'') + \frac{1}{26}\omega(G'') + p \leq \frac{25}{26}\Delta(G) + \frac{1}{26} \omega(G)$.

Thus we may suppose that $\Delta(G'') < \Delta_\theDeltaSparsity$. In this case, we have
\begin{align*} 
\chi(G) &\leq  \Delta_\theDeltaSparsity + p \\ 
&\leq \Delta_\theDeltaSparsity+\omega(G) \\
&\leq \Delta_\theDeltaSparsity + (1-\varepsilon)\cdot \omega(G)+ \varepsilon \cdot \omega(G).
\end{align*}
Finally, by the assumption on $\omega(G)$, we have
\begin{align*} 
\chi(G) &\leq \Delta_\theDeltaSparsity + (1-\varepsilon)\cdot (\Delta(G)-2 \Delta_\theDeltaSparsity)+ \varepsilon \cdot \omega(G)\\
&= (1-\varepsilon)\cdot \Delta(G)+ \varepsilon \cdot \omega(G)+(2\varepsilon -1)\cdot \Delta_\theDeltaSparsity\\
&\leq (1-\varepsilon)\cdot \Delta(G)+ \varepsilon \cdot \omega(G),
\end{align*}
hence the theorem holds.

\end{proof}

\section{Acknowledgements}
Most of this work was done while the first author was a postdoc and the second author was a visiting student at the University of Waterloo.
\bibliographystyle{plain}
\bibliography{references}
\end{document}